\UseRawInputEncoding
	\documentclass{article}
         \usepackage{bm}
	\usepackage{amsmath}
	\usepackage{amsthm}
	\usepackage{amsfonts}
	\usepackage{multirow}
    \usepackage{nicefrac}
    \usepackage{color}
    \usepackage{float}
    \usepackage{graphicx, amssymb,graphics}
\usepackage{cite}
\usepackage{hyperref}
\usepackage{cleveref}

\newtheorem{theorem}{Theorem}[section]
\newtheorem{lemma}[theorem]{Lemma}
\newtheorem{Remark}[theorem]{Remark}
\makeatletter
\@addtoreset{equation}{section}
\makeatother

\newcommand{\cN}{\mbox{{${\cal N}$}}}
\newcommand{\cL}{\mbox{{${\cal L}$}}}
\newcommand{\cD}{\mbox{{${\cal D}$}}}
\newcommand{\cK}{\mbox{{${\cal K}$}}}

\begin{document}
		\title{Energy stable arbitrary order ETD-MS method for gradient flows with Lipschitz nonlinearity}
		
		\author{
			Wenbin Chen\thanks{%EMW lab and
				Shanghai Key Laboratory for Contemporary Applied Mathematics,
				School of Mathematical Sciences; Fudan University, Shanghai, China 200433 ({\tt wbchen@fudan.edu.cn})}
			\and
			Shufen Wang\thanks{%EMW lab and
				School of Mathematical Sciences; Fudan University, Shanghai, China 200433 ({\tt 17110180015@fudan.edu.cn})}
			\and
			Xiaoming Wang\thanks{SUSTech International Center for Mathematics and Department of Mathematics and  and Guangdong Provincial Key Laboratory of Computational Science and Material Design and National Center for Applied Mathematics Shenzhen, Southern University of Science and Technology, Shenzhen, China 518055, PRC  ({\tt wangxm@sustech.edu.cn}), corresponding author.}
		}	
	
%	\author{Wenbin Chen		\and Shufen Wang	\and Xiaoming Wang}
	
%	\authorrunning{Short form of author list} % if too long for running head
	
%	\institute{Wenbin Chen \at
%	              Shanghai Key Laboratory of Mathematics for Nonlinear Sciences, School of Mathematical Sciences; Fudan University, Shanghai 200433, China. \\
%	              \email{wbchen@fudan.edu.cn}           %  \\
%	%             \emph{Present address:} of F. Author  %  if needed
%	           \and
%	           Shufen Wang \at
%	              School of Mathematical Sciences; Fudan University, Shanghai 200433, China. \\
%	              \email{17110180015@fudan.edu.cn}
%	           \and
%	           Xiaoming Wang \at
%	              Department of Mathematics and SUSTech International Center for Mathematics, Southern University of Science and Technology, Shenzhen 518055, P.R. China.\\
%	              \email{wangxm@sustech.edu.cn}
%	}
	
%			\date{Received: date / Accepted: date\\}
	\maketitle
			
		\begin{abstract}
			We present a methodology to construct efficient high-order in time accurate numerical schemes for a class of gradient flows with appropriate Lipschitz continuous nonlinearity. There are several ingredients to the strategy: the exponential time differencing (ETD), the multi-step (MS) methods, the idea of stabilization, and the technique of interpolation. They are synthesized  to develop a generic $k^{th}$ order in time efficient linear numerical scheme with the help of an artificial regularization term of the form $A\tau^k\frac{\partial }{\partial t}\cL^{p(k)} u$ where $\cL$ is the positive definite linear part of the flow, $\tau$ is the uniform time step-size.  The exponent $p(k)$ is determined explicitly by the strength of the Lipschitz nonlinear term in relation to $\cL$ together with the desired temporal order of accuracy $k$. To validate our theoretical analysis,  the thin film epitaxial growth without slope selection model is examined with a fourth-order ETD-MS discretization in time { and} Fourier pseudo-spectral in space discretization. Our numerical results on  convergence and energy stability are in accordance with our theoretical results.
		\end{abstract}
	
		\section{Introduction} \label{sec1}
%		{
		
		Many natural and engineering processes are gradient flows in the sense that the time evolution of the system is in the direction of decreasing certain energy functional associated with the state of the system. They have a wide range of applications in materials science, fluid dynamics \cite{allen1979microscopic,anderson1998diffuse,cahn1958free,doi1988theory,elder2002modeling,gurtin1996two,leslie1979theory,yue2004diffuse} as well as in geometry (geometric flows) and PDEs (optimal transportation) \cite{AmbrosioGigliSavare2008} among many others.
		The evolution of these gradient flows could be complicated. Efficient and accurate numerical methods are highly desirable, especially in the generic case of the absence of solution formula. In addition, the evolution process could be long before it settles to certain equilibrium state(s), see for instance \cite{kohn2003upper}. Hence, it  is of great importance to design numerical methods that inherits the energy law of the gradient flow, even if in a slightly modified form,  if one is interested in the long evolution process such as the coarsening process associated with many phase field models at the  large system size regime.

		An abundant work exists in the development of (energy) long-time stable schemes  and numerical simulation of various gradient flows arising in material science and fluid dynamics using convex splitting, truncation, SAV, and IEQ method among others, see for instance  \cite{wang2010Unconditionally,shen2012second,elliott1993global,eyre1998unconditionally,shen2010numerical,shin2017unconditionally,shen2018scalar,shen2019new,gong2020arbitrarily}, and the references therein.
		Exponential time differencing (ETD) is a very appealing time discretization method which achieves its high-order accuracy in time with exact treatment of the linear part \cite{beylkin1998new,du2004stability,du2005analysis} together with Duhammel's principle applied to the nonlinear term. The introduction of integrating factor gives rise to a nonlinear integral term.  There are two popular approaches in approximating the nonlinear part: Runge-Kutta (RK) method \cite{cox2002exponential,hochbruck2005explicit,hochbruck2010exponential} and multi-step(MS) method \cite{hochbruck2010exponential,hochbruck2011exponential,hairer1993n}.  Abundant applications of these two methods to various gradient models can be found in \cite{Asamani2020second,chen2020stabilized, chen2020energy,cheng2019third, Du2019Maximum, huang2019fast, ju2014, ju2015a, ju2015b, wang2016Efficient} . The approximations are usually explicit in order to preserve the efficiency of the ETD method.
		While it is relatively straightforward to construct numerical  schemes of arbitrary order formally via RK or MS method, the energy stability of the algorithms are nontrivial since the explicit treatment of the nonlinear term induces instability, and the existing works involving rigorous energy stability analysis for high accuracy scheme are limited. The authors in \cite{ju2018energy} proved the energy stability only for the first order ETD scheme for the thin film epitaxial growth model without slope selection (NSS).
		In a recent work,  the idea of stabilization was utilized to develop a second order in time ETD-MS method that is energy stable with the aid of a judiciously chosen stabilizing term of matching order for the NSS equation \cite{chen2020stabilized},  it's the first work to provide both the energy stability and convergence results for phase field models theoretically. This idea has been developed further in \cite{chen2020energy} for a third order in time ETD-MS energy stable scheme for the NSS model, and to even higher orders without detailed proof. Another third-order stabilized energy stable ETD-MS scheme proposed in \cite{cheng2019third} to approximate the NSS equation also gave a rigorous energy stability analysis, in which a different form of stabilized term $A\tau^2\Delta_N^2(u^{n+1}-u^n)$ was added with $A=\mathcal{O}(\epsilon^{-2})$ to guarantee energy stability.
		The purpose of this manuscript is to present a systematic approach to construct ETD-MS based energy stable schemes of arbitrary order in time with the help of an appropriate stabilizing term for a class of gradient flows on a Hilbert space with a positive linear part and a mild nonlinear part satisfying Lipschitz condition in some suitable sense. More importantly, we can make the stabilized coefficient $A$ independent of the small parameter $\epsilon$ or time step-size $\tau$ (i.e., of order $\mathcal{O}(1)$) by a proper choice of $p(k)$ (the spatial order of stabilized term). So far as we know, this is the first result on unconditionally energy stable, arbitrary-order ETD-MS based efficient algorithm for this type of gradient flow.
		%}
		
		Let $E(u)$ be a suitable energy functional on a Hilbert space $H$ (with the domain being a subspace of $H$), and let $\frac{\delta E}{\delta u}$ be the variational derivative of $E$ with respect to $u$ in $H$. The associated gradient flow with mobility $M\ge 0$ can then be formulated as follows
		\begin{align} \label{eq:gradient flow}
			\frac{\partial u}{\partial t}  = -M\frac{\delta E}{\delta u} .
	%		=  - \epsilon \cL u  + \cN(u),
		\end{align}
		%{
		Taking inner product with \eqref{eq:gradient flow} and $\frac{\delta E}{\delta u}$ on space $H$ we formally arrive at the following  energy equality:
		\begin{align} \label{eq:energy decay GF}
			\frac{dE(u)}{dt} = -M\left\|\frac{\delta E}{\delta u}\right\|^2_H\le 0.
		\end{align}
		This implies that the flow always evolves in the direction of decreasing energy. %}
		Notice that both the classical Allen-Cahn model and the Cahn-Hilliard model fall into this framework with the same Ginzburg-Landau free energy but different Hilbert space $H$.
		
		The variational derivative of $E$ can be split into two parts, a linear part $\cL u$, and a nonlinear one $\cN(u)$.  In additional, large system size for models arising in material sciences often corresponds the existence of a small parameter in front of the linear term once we non-dimensionalize the system. Formulating the system as an abstract ODE in the Hilbert space $H$ or an ODE system after performing appropriate spatial discretization (say Fourier pseudo-spectral in the spatially periodic case) we have
		\begin{align} \label{eq:gf}
		\frac{d u}{d t} =  - \epsilon \cL u  + \cN(u).
		\end{align}
		
		We will impose the following two assumptions on \eqref{eq:gf} :
		\begin{enumerate}
			\item  The operator $\cL$ is non-negative. \\
			Thus we can define operators  $\cL^{\alpha/2}$ for any $\alpha\ge0$. The domain of operator $\cL^{\alpha/2}$ is denoted by $V^{\alpha} $. For $\alpha=1$ and $\alpha=0$,  it is abbreviated as $\cD(\cL^{\frac12}) = V$ and $\cD(\cL^0) = H$, respectively;
			\item The nonlinear term is Lipschitz continuous in the sense that: $\exists \beta>0, \gamma>0,C_L>0$ such that
			\begin{align} \label{eq:LIP}
				\left\| \cN(u) - \cN(v) \right\|_{V^{-\beta}} \le C_L \left\| u-v\right\|_{V^\gamma}, \forall u, v \in V^\gamma,
			\end{align}
			{where $V^{-\beta}$ is the dual space to $V^{\beta}$ with the duality induced by the inner product on $H$}.
		\end{enumerate}
		The case when $\cL$ is only semi{-}positive definite or positive definite minus a constant multiply of the identity operator can be treated after we shift the $\cL$ so that it is positive definite, and modify the ``nonlinear'' term accordingly (now it may contain a linear part).
		
	A typical gradient flow that fits into this abstract framework is the thin film epitaxy growth equation without slope selection (see  \cite{golubovic1997interfacial,moldovan2000interfacial,kohn2003upper,li2003thin,li2004epitaxial} and references therein):
	\begin{align} \label{eq:nss}
	\frac{\partial u}{\partial t}=-\epsilon \Delta^{2} u-\nabla \cdot\left(\frac{\nabla u}{1+|\nabla u|^{2}}\right)
	\end{align}
	with the energy functional given by
	\begin{align} \label{eq:SS energy functional}
		E(u)=\int_{\Omega} \left( \frac{\epsilon}{2}|\Delta u|^{2}-\frac{1}{2} \ln \left(1+|\nabla u|^{2}\right)  \right) \mathrm{d} \mathbf{x},
	\end{align}
	and the Hilbert space $H=L^2$, the linear and nonlinear operators are $\cL = \Delta^2$, $\cN(u) = -\nabla\cdot \left( \frac{\nabla u}{1+|\nabla u|^2}\right)$.
	
	Most gradient flows do not fit into this framework directly. However, many of them enjoy invariant regions in the phase space that attract all solutions such as the Cahn-Hilliard equation, the Allen-Cahn equation and many other reaction-diffusion equations, and the two dimensional Navier-Stokes equations \cite{Temam1997Infinite} among others. For these models, we could modify the nonlinear term with the help of an appropriate truncation to derive a ``prepared'' equation as in the theory of inertial manifolds so that certain Lipschitz continuity condition is satisfied \cite{Temam1997Infinite}. Hence the framework is quite general.
	
	The main contribution of this manuscript is the construction and energy stability analysis of ETD-MS based numerical schemes of arbitrary order for \eqref{eq:gf} under the two assumptions postulated above, with an appropriate stabilization term of the form $A\tau^k \frac{d}{dt} \cL^{p(k)} u(t) $ for $k^{th}$ order scheme where $A$ is independent of the small parameter $\epsilon$ or the spatial/temporal discretization. The exponent $p(k)$ is determined explicitly independent of $\epsilon$ or the temporal/spatial grids. These algorithms are highly efficient since only one fixed positive definite constant coefficient elliptic problem needs to be solved at each time step. So far as we know, this is the first time such an efficient systematic approach is presented although it was alluded to in \cite{chen2020energy}.
	
	The rest of the manuscript is organized as follows. We present the numerical scheme together with the energy stability analysis in section 2. Numerical results are presented in section 3. Concluding remarks are offered in section 4.
	
%%%%%	
	\section{Numerical scheme} \label{sec3}
	In this section, we propose a temporal semi-discrete scheme for \eqref{eq:gf}, in which a generic $k^{th}$ order approximation in time is constructed by applying the exponential time differencing and multi-step method.  The nonlinear term $\cN(u)$ is treated explicitly and the Lagrange approximation is adopted, leading to a  linear system.  However, higher order multi-step treatment to nonlinear term gives  rise to strong instability, thus to keep the energy decaying property, a $k^{th}$ order regularization term of the form of $A\tau^k \frac{d}{dt} \cL^{p(k)} u(t) $ is added with a careful choice of exponent $p(k)$.

\subsection{The algorithm}	
	Denote the discrete numerical solution at $t=t^n$ by $u^n$. The differential form of the numerical scheme for \eqref{eq:gf} is to find $u^{n+1}(t):[t_n,t_{n+1}]\rightarrow$ an appropriate subspace of $H$ such that
	\begin{align} \label{eq:etdmsns}
			\frac{du^{n+1}(t)}{dt} + \epsilon \cL u^{n+1}(t) + A\tau^k \frac{d}{dt} \cL^{p(k)} u^{n+1}(t) = \sum_{i=0}^{k-1} \ell_i(t-t_n) \cN(u^{n-i}),
	\end{align}
	where  $\ell_i(s)$ are the shifted (to the negative range) Lagrange basis polynomial of degree $k$ with the form of
	\begin{align} \label{eq:Lagrange basis}
		\ell_i(s) = \prod_{0\le m\le k-1 \atop m\neq i} \frac{m\tau+s}{(m-i)\tau} =  \prod_{0\le m\le k-1 \atop m\neq i} \frac{m +s/\tau}{m-i}  =  \sum_{j=0}^{k-1} \xi_{i,j} s^j,
	\end{align}
	where $\{ \xi_{i,j} \}_{j=0}^{k-1}$ are coefficients of the polynomial $\ell_{i}(s)$. Obviously, $\xi_{i,j} \sim \mathcal{O}(\tau^{-j})$. This property will be used later in the stability analysis. %, in addition, $ \xi_{i,0}=0,~\forall 1\le i \le k-1$ and $\xi_{0,0}=1$.
	
	Introducing an integrating factor $e^{\cK t}$ with $\cK=\epsilon\left(I+A\tau^k\cL^{p(k)}\right)^{-1}\cL$ and integrating the equation \eqref{eq:etdmsns} from $t^n$ to $t^{n+1}$ leads to the following $k^{th}$ order ETD-MS scheme
	\begin{align} \label{eq:etdmsns 1}
		u^{n+1} =& e^{-\cK \tau}u^n + \sum_{i=0}^{k-1}\int_{0}^{\tau} e^{-\cK (\tau -s)} \ell_i(s)  ds\, \cN(u^{n-i}).
	\end{align}
	Since the Lagrange interpolation polynomials and the linear operator $\cL$ are known a priori, the algorithm can be simplified. Let $\phi_j:=\int_{0}^{\tau} e^{-\cK (\tau -s)}  s^j ds$ which can be calculated (explicitly) before hand,  it can be shown via recurrence formula
	\begin{equation} \label{eq:phi recur}
	\left\{
		\begin{aligned}
		\phi_0 = & \cK^{-1}\left(I - e^{-\cK\tau}\right), \\
		\phi_j = & \cK^{-1}\left(\tau^j -j\phi_{j-1} \right), \quad 1\le j \le k-1.
		\end{aligned} \right.
	\end{equation}
	Thus the $k^{th}$ order ETD-MS scheme \eqref{eq:etdmsns 1} can be rewritten as
	\begin{align} \label{eq:etdmsns 2}
	u^{n+1} = e^{-\cK \tau}u^n + \sum_{i=0}^{k-1}\sum_{j=0}^{k-1} \xi_{i,j}\phi_j\,\cN(u^{n-i}).
	\end{align}
	This is an extremely efficient algorithm, especially since we could pre-calculate the operators involved.

	\subsection{Energy stability} \label{sec31}
	In this subsection, we establish the energy stability for the scheme \eqref{eq:etdmsns}. First we present an interpolation estimation that will be used later.
	\begin{lemma} \label{lem:interpolation}
		For any $u\in V^{\beta}$ and $v\in V^{\gamma}$, $p(k)>\max\{\beta,\gamma\}, q\in (0,1)$, and any constants $\hat{C}, \tilde{C}$ to be determined for specific problem, the following inequality holds
		\begin{align} \label{eq:bad term}
		\tau \left\| \frac{du}{dt}\right\|_{V^\beta} \left\| \frac{dv}{dt}\right\|_{V^\gamma} \le C_1 \left\| \frac{du}{dt}\right\|_{H}^{2}  + C_2 \tau^{\frac{2qp(k)}{\beta}} \left\|\frac{du}{dt}\right\|_{V^{p(k)}}^{2} + C_3 \left\| \frac{dv}{dt}\right\|_{H}^{2} + C_4 \tau^{\frac{2(1-q)p(k)}{\gamma}} \left\| \frac{dv}{dt}\right\|_{V^{p(k)}}^{2},
		\end{align}
		where
		\begin{equation} \label{eq:c1c2c3c4}
			\begin{aligned}
			C_1(\hat{C}) &= \frac{1-\beta/p(k)}{2} \hat{C}^{1/(1-\beta/p(k))}, \quad C_2(\hat{C}) = \frac{\beta}{2p(k)}\hat{C}^{-p(k)/\beta}, \\
			C_3(\tilde{C}) &= \frac{1-\gamma/p(k)}{2} \tilde{C}^{1/(1-\gamma/p(k))}, \quad C_4(\tilde{C}) = \frac{\gamma}{2p(k)}\tilde{C}^{-p(k)/\gamma}.
			\end{aligned}
		\end{equation}
%		{ and $\hat{C},~\tilde{C}$ are generic constants that will be determined in specific problem.}
	\end{lemma}
	\begin{proof}
		Using the interpolation inequality to control $\|\cdot\|_{V^\beta}$ by $\|\cdot\|_{H}$ and $\|\cdot\|_{V^{p(k)}}$:
		\begin{align} \label{eq:interpolation}
		\tau \left\| \frac{du}{dt}\right\|_{V^\beta} \left\| \frac{dv}{dt}\right\|_{V^\gamma} \le \tau \left\| \frac{du}{dt}\right\|_{H}^{1-\beta/p(k)}  \left\| \frac{du}{dt}\right\|_{V^{p(k)}}^{\beta/p(k)} \cdot \left\| \frac{dv}{dt}\right\|_{H}^{1-\gamma/p(k)}  \left\| \frac{dv}{dt}\right\|_{V^{p(k)}}^{\gamma/p(k)} .
		\end{align}
		Denote the term in the right hand side (RHS) of \eqref{eq:interpolation} by $I_1$, then we use the Young's inequality twice to show
		\begin{align} \label{eq:bad term2}
		I_1 \le &  \frac12 \tau^{2q} \left\| \frac{du}{dt}\right\|_{H}^{2-2\beta/p(k)}  \left\| \frac{du}{dt}\right\|_{V^{p(k)}}^{2\beta/p(k)}  + \frac12 \tau^{2-2q} \left\| \frac{dv}{dt}\right\|_{H}^{2-2\gamma/p(k)}  \left\| \frac{dv}{dt}\right\|_{V^{p(k)}}^{2\gamma/p(k)} \nonumber \\
		\le & \frac{1-\beta/p(k)}{2}\left(\hat{C} \left\| \frac{du}{dt}\right\|_{H}^{2-2\beta/p(k)}  \right)^{1/(1-\beta/p(k))}  + \frac{\beta}{2p(k)}\left(  \frac{\tau^{2q}}{\hat{C}}\left\| \frac{du}{dt}\right\|_{V^{p(k)}}^{2\beta/p(k)} \right)^{p(k)/\beta} \nonumber \\
		& + \frac{1-\gamma/p(k)}{2}\left( \tilde{C}\left\| \frac{dv}{dt}\right\|_{H}^{2-2\gamma/p(k)}  \right)^{1/(1-\gamma/p(k))}  + \frac{\gamma}{2p(k)}\left( \frac{\tau^{2-2q}}{\tilde{C}} \left\| \frac{dv}{dt}\right\|_{V^{p(k)}}^{2\gamma/p(k)} \right)^{p(k)/\gamma} \nonumber \\
		= & C_1 \left\| \frac{du}{dt}\right\|_{H}^{2}  + C_2 \tau^{\frac{2qp(k)}{\beta}} \left\|\frac{du}{dt}\right\|_{V^{p(k)}}^{2} + C_3 \left\| \frac{dv}{dt}\right\|_{H}^{2} + C_4 \tau^{\frac{2(1-q)p(k)}{\gamma}} \left\| \frac{dv}{dt}\right\|_{V^{p(k)}}^{2}.
		\end{align}
		This completes the proof of Lemma~\ref{lem:interpolation}.
	\end{proof}
	
	Now we come to the energy stability. For simplicity, we denote $\|\cdot\|_{L^2(t_i,t_j;V^{\alpha})}$ by $\|\cdot\|_{L^2(I_{i,j};V^{\alpha})}$ hereafter.
	\begin{lemma} \label{lem:energy estimate1}
		For system \eqref{eq:etdmsns}, the following energy estimation establishes:
		\begin{align} \label{eq:energy estimate 1}
		&E(u^{n+1}) - E(u^{n})+	\left\| 	\frac{du^{n+1}(t)}{dt} \right\|_{L^2(I_{n,n+1};H)}^2 + A\tau^k \left\| 	\frac{du^{n+1}(t)}{dt}\right\|_{L^2(I_{n,n+1};V^{p(k)})}^2  \nonumber \\
		\le & \sum_{j=0}^{k-1}C_L  \tau^{\frac12}\left\| 1- \sum_{i=-1}^{j-1}\ell_i(t-t_n) \right\|_{L^2(I_{n,n+1})} \left\| \frac{du^{n-j+1}(t)}{dt} \right\|_{L^2(I_{n-j,n-j+1};V^{\gamma})} \left\| \frac{du^{n+1}(t)}{dt} \right\|_{L^2(I_{n,n+1};V^{\beta})},
		\end{align}
		where the convention $\ell_{-1}(t-t_n)=0$ has been used.
	\end{lemma}
	\begin{proof}
		To establish the desired energy estimates, we take inner product of \eqref{eq:etdmsns} with $	\frac{du^{n+1}(t)}{dt}$, which gives
		\begin{align} \label{eq:energy stability1}
		&	\left\| 	\frac{du^{n+1}(t)}{dt} \right\|_H^2 + A\tau^k \left\| 	\frac{du^{n+1}(t)}{dt}\right\|_{V^{p(k)}}^2 + \frac{d}{dt} E(u^{n+1}(t)) \nonumber \\
		=& \left( \sum_{i=0}^{k-1} \ell_i(t-t_n) \cN(u^{n-i}) - \cN(u^{n+1}(t)), \frac{du^{n+1}(t)}{dt}\right)_H.
		\end{align}
		Integrating from $t_n$ to $t_{n+1}$ gives
		\begin{align} \label{eq:energy stability2}
		&	\left\| 	\frac{du^{n+1}(t)}{dt} \right\|_{L^2(I_{n,n+1};H)}^2 + A\tau^k \left\| 	\frac{du^{n+1}(t)}{dt}\right\|_{L^2(I_{n,n+1};V^{p(k)})}^2 + E(u^{n+1}) - E(u^{n}) \nonumber \\
		=& \int_{t_n}^{t_{n+1}}\left( \sum_{i=0}^{k-1} \ell_i(t-t_n) \cN(u^{n-i}) - \cN(u^{n+1}(t)), \frac{du^{n+1}(t)}{dt}\right)_H dt.
		\end{align}
		
		Note that the sum of the Lagrange basis functions equals one, i.e., $\sum_{i=0}^{k-1} \ell_i(t-t_n) = 1$, thus terms within the integral in RHS of \eqref{eq:energy stability1} (denoted by NLT) can be rewritten as
		\begin{align} \label{eq:NLT}
		\text{NLT} = &  \left( \sum_{i=0}^{k-1} \ell_i(t-t_n) 　 \left( \cN(u^{n-i}) - \cN(u^{n+1}(t)) \right), \frac{du^{n+1}(t)}{dt}\right)_H \nonumber \\
		= &  \left( \sum_{i=0}^{k-1} \ell_i(t-t_n) 　 \left( \cN(u^{n-i}) - \cN(u^{n-i+1}) + \cdots+ \cN(u^{n})  -\cN(u^{n+1}(t)) \right), \frac{du^{n+1}(t)}{dt}\right)_H \nonumber \\
		= & \left( \sum_{i=0}^{k-1} \ell_i(t-t_n) 　 \left( \cN(u^{n}) - \cN(u^{n+1}(t)) \right), \frac{du^{n+1}(t)}{dt}\right)_H \nonumber \\
		& + \sum_{j=1}^{k-1}\left(  \sum_{i=j}^{k-1}\ell_{i}(t-t_n) 　 \left( \cN(u^{n-j}) - \cN(u^{n-j+1}) \right), \frac{du^{n+1}(t)}{dt}\right)_H \nonumber \\
		= & \left(   \cN(u^{n}) - \cN(u^{n+1}(t)), \frac{du^{n+1}(t)}{dt}\right)_H \nonumber \\
		& +\sum_{j=1}^{k-1} \left( \left(1- \sum_{i=0}^{j-1}\ell_i(t-t_n) \right)　　 \left( \cN(u^{n-j}) - \cN(u^{n-j+1}) \right), \frac{du^{n+1}(t)}{dt}\right)_H.
		\end{align}
		Using the Cauchy-Schwartz inequality and the Lipschitz continuity \eqref{eq:LIP} makes
		\begin{align}  \label{eq:NLT1}
		&\int_{t_n}^{t_{n+1}} \left(   \cN(u^{n}) - \cN(u^{n+1}(t)), \frac{du^{n+1}(t)}{dt}\right)_H dt \nonumber \\
		\le & C_L \int_{t_n}^{t_{n+1}} \left\| u^n -u^{n+1}(t) \right\|_{V^{\gamma}} \left\| \frac{du^{n+1}(t)}{dt} \right\|_{V^{\beta}} dt \nonumber \\
		\le & C_L \int_{t_n}^{t_{n+1}} \tau^{\frac12}\left\| \frac{du^{n+1}(t)}{dt} \right\|_{L^2(I_{n,n+1};V^{\gamma})} \left\| \frac{du^{n+1}(t)}{dt} \right\|_{V^{\beta}} dt \nonumber \\
		\le & C_L  \tau \left\| \frac{du^{n+1}(t)}{dt} \right\|_{L^2(I_{n,n+1};V^{\gamma})} \left\| \frac{du^{n+1}(t)}{dt} \right\|_{L^2(I_{n,n+1};V^{\beta})} \nonumber \\
		:= & \mathrm{NLT_0},
		\end{align}
		where the last inequality follows from the H\"{o}lder inequality.
		
		Similarly, the remaining terms in RHS of \eqref{eq:energy stability2} can be estimated as
		\begin{align} \label{eq:NLT2}
		&	\int_{t_n}^{t_{n+1}} \left(\left(1- \sum_{i=0}^{j-1}\ell_i(t-t_n) \right) \left( \cN(u^{n-j}) - \cN(u^{n-j+1}) \right), \frac{du^{n+1}(t)}{dt}\right)_H dt \nonumber \\
		\le & C_L  \tau^{\frac12} \left\| \frac{du^{n-j+1}(t)}{dt} \right\|_{L^2(I_{n-j,n-j+1};V^{\gamma})} \left\| \frac{du^{n+1}(t)}{dt} \right\|_{L^2(I_{n,n+1};V^{\beta})}  \left\| 1- \sum_{i=0}^{j-1}\ell_i(t-t_n) \right\|_{L^2(I_{n,n+1})} \nonumber \\
		:= &\mathrm{NLT_j}, \quad 1\le j \le k-1.
		\end{align}
		This completes the proof.
	\end{proof}

	Next we give an upper bound for the $L^2$-integral in time $ \left\| 1- \sum_{i=0}^{j-1}\ell_i(t-t_n) \right\|_{L^2(I_{n,n+1})}$ and further provide the energy stability for scheme \eqref{eq:etdmsns}.
	Recall that in \eqref{eq:Lagrange basis}, the Lagrange basis $\ell_i(s)$ is expressed as the polynomial of $s$ with coefficients $\xi_{i,j}$. According to the properties of $\xi_{i,j}$, it's easy to see
	\begin{align} \label{eq:ell L2 integral}
		\left\| 1- \sum_{i=0}^{j-1}\ell_i(t-t_n) \right\|_{L^2(I_{n,n+1})} = 	\left\| 1- \sum_{i=0}^{j-1}\sum_{r=0}^{k-1} \xi_{i,r} (t-t_n)^r \right\|_{L^2(I_{n,n+1})} = C^{*}_j\tau^{1/2}, \quad 1\le j\le k-1,
	\end{align}
	%{
	where the constants $C^{*}_j$ are independent of time step-size $\tau$ or the current time $t$. For convenience, we follow the convention of $C_0^*=1$ hereafter.
	
	We now introduce the following notation for the sake of brevity in presentation.
	\begin{align} \label{eq:cjbar}
		 \bm{\overline{C}}_j := \sum_{r=0}^{k-1-j}C_{k-1-r}^*, \quad j=0, \cdots, k-1 .
	\end{align}
	It follows that
	\begin{equation}\label{eq:cjbar-rel}
	\bm{\overline{C}}_j = \bm{\overline{C}}_{j+1} +C_j^*, \quad \bm{\overline{C}}_{k-1} =C^*_{k-1}.
	\end{equation}
	Next, we define the following modified energy
	\begin{align} \label{eq:modifed energy}
	\tilde{E}(u^n) =& E(u^n) + C_LC_3 \sum_{j=1}^{k-1} \bm{\overline{C}}_j\left\| \frac{du^{n-j+1}(t)}{dt}\right\|_{L^2(I_{n-j,n-j+1};H)}^{2} \nonumber \\
	& + C_LC_4 \sum_{j=1}^{k-1} \bm{\overline{C}}_j \tau^k \left\| \frac{du^{n-j+1}(t)}{dt}\right\|_{L^2(I_{n-j,n-j+1};V^{p(k)})}^{2},
	\end{align}
	where both $C_3, ~C_4$  depend on $\hat{C}$, $\tilde{C}$ as specified  in \eqref{eq:c1c2c3c4}.
	
Thanks to \eqref{eq:c1c2c3c4},  $C_1, C_3$ can be made as small as we need so long as we set $\hat{C}, \tilde{C}$ small enough. Therefore, for
small enough  constants $\hat{C}$, $\tilde{C}$ and large enough constant $A$ the following inequalities hold
	\begin{align}
	\left( 1 - C_L \left(C_3+C_1\bm{\overline{C}}_0 \right)\right) \ge C_LC_3\bm{\overline{C}}_1 , \label{eq:constant choose1} \\
	\left( A - C_L\left(C_4+C_2\bm{\overline{C}}_0 \right)\right) \ge C_LC_4\bm{\overline{C}}_1. \label{eq:constant choose2}
	\end{align}
	Note that $C_0^*=1$, and hence $\bm{\overline{C}}_0=\bm{\overline{C}}_1 + 1$ according to \eqref{eq:cjbar-rel}. Therefore \eqref{eq:constant choose1}--\eqref{eq:constant choose2} are simplified as
	\begin{align*} 
	1 & \ge C_{L} \left(C_3+C_1 \right)\bm{\overline{C}}_0  , \\
	A &\ge C_{L}\left(C_4+C_2 \right)\bm{\overline{C}}_0 .
	\end{align*}
	Pick $\hat{C}$ and $\tilde{C}$ so that $C_{L} \left(C_3+C_1 \right)\bm{\overline{C}}_0 \le 1$, i.e., $(1-\beta/p(k))\hat{C}^{1/(1-\beta/p(k))} + (1-\gamma/p(k))\tilde{C}^{1/(1-\gamma/p(k))} \le 2/(C_{L}\bm{\overline{C}}_0)$, and then let
	\begin{align*}
	A = &  C_{L}\left(\frac{\beta}{2p(k)}\hat{C}^{-p(k)/\beta}+\frac{\gamma}{2p(k)}\tilde{C}^{-p(k)/\gamma} \right)\bm{\overline{C}}_0,
	\end{align*}
	we have \eqref{eq:constant choose1}--\eqref{eq:constant choose2}. 
	
	We are now ready to prove the main result of the energy stability.
	%}
	\begin{theorem} \label{thm:energy stability}
			The numerical scheme \eqref{eq:etdmsns} is energy stable in the sense that
			\begin{align} \label{eq:energy stable}
			\tilde{E}(u^{n+1}) \le \tilde{E}(u^n), \quad \forall n\ge k,
			\end{align}
			provided that \eqref{eq:constant choose1}--\eqref{eq:constant choose2} are satisfied, and $p(k)=\frac{(\beta+\gamma)k}{2}$.
	\end{theorem}
	\begin{proof}
	By \eqref{eq:ell L2 integral}, the estimation \eqref{eq:NLT2} for $\mathrm{NLT_{j}}, ~1\le j\le k-1$ can be simplified to
	\begin{align} \label{eq:NLT3}
	\mathrm{NLT_{j}} =  C_L C_j^* \tau \left\| \frac{du^{n-j+1}(t)}{dt} \right\|_{L^2(I_{n-j,n-j+1};V^{\gamma})} \left\| \frac{du^{n+1}(t)}{dt} \right\|_{L^2(I_{n,n+1};V^{\beta})} .
	\end{align}
	Applying Lemma \ref{lem:interpolation} to \eqref{eq:NLT1} and \eqref{eq:NLT3}, these nonlinear terms can be bounded further:
	\begin{align}
	\mathrm{NLT_0} \le & C_L \left[ C_1 \left\| \frac{du^{n+1}(t)}{dt}\right\|_{L^2(I_{n,n+1};H)}^{2}  + C_2 \tau^{\frac{2qp(k)}{\beta}} \left\|\frac{du^{n+1}(t)}{dt}\right\|_{L^2(I_{n,n+1};V^{p(k)})}^{2} \right. \nonumber \\
	& \left. + C_3 \left\| \frac{du^{n+1}(t)}{dt}\right\|_{L^2(I_{n,n+1};H)}^{2} + C_4 \tau^{\frac{2(1-q)p(k)}{\gamma}} \left\| \frac{du^{n+1}(t)}{dt}\right\|_{L^2(I_{n,n+1};V^{p(k)})}^{2} \right],  \label{eq:interp1}\\
	\mathrm{NLT_{j}} \le & C_L C_j^*\left[ C_1\left\| \frac{du^{n+1}(t)}{dt}\right\|_{L^2(I_{n,n+1};H)}^{2} + C_2 \tau^{\frac{2qp(k)}{\beta}} \left\|\frac{du^{n+1}(t)}{dt}\right\|_{L^2(I_{n,n+1};V^{p(k)})}^{2} \right. \nonumber \\
	& \left. + C_3\left\| \frac{du^{n-j+1}(t)}{dt}\right\|_{L^2(I_{n-j,n-j+1};H)}^{2} + C_4\tau^{\frac{2(1-q)p(k)}{\gamma}} \left\| \frac{du^{n-j+1}(t)}{dt}\right\|_{L^2(I_{n-j,n-j+1};V^{p(k)})}^{2} \right]. \label{eq:interp2}
	\end{align}	
	Choosing indexes $q$ and $p(k)$  to satisfy
	\begin{align} \label{eq:p(k)}
	\frac{2qp(k)}{\beta} = k, \quad \frac{2(1-q)p(k)}{\gamma} = k,
	\end{align}
	and simple calculation shows
	\begin{align} \label{eq:p(k)2}
			q = \frac{1}{1+\gamma/\beta}, \quad	p(k) = \frac{(\beta+\gamma)k}{2}.
	\end{align}
	%{
	Then estimates \eqref{eq:interp1}--\eqref{eq:interp2} give
		\begin{align}
	\mathrm{NLT_0} \le & C_L \left[ C_1 \left\| \frac{du^{n+1}(t)}{dt}\right\|_{L^2(I_{n,n+1};H)}^{2}  + C_2 \tau^{k} \left\|\frac{du^{n+1}(t)}{dt}\right\|_{L^2(I_{n,n+1};V^{p(k)})}^{2} \right. \nonumber \\
	& \left. + C_3 \left\| \frac{du^{n+1}(t)}{dt}\right\|_{L^2(I_{n,n+1};H)}^{2} + C_4 \tau^{k} \left\| \frac{du^{n+1}(t)}{dt}\right\|_{L^2(I_{n,n+1};V^{p(k)})}^{2} \right],  \label{eq:interp3}\\
	\mathrm{NLT_{j}} \le & C_L C_j^*\left[ C_1\left\| \frac{du^{n+1}(t)}{dt}\right\|_{L^2(I_{n,n+1};H)}^{2}  + C_2 \tau^{k} \left\|\frac{du^{n+1}(t)}{dt}\right\|_{L^2(I_{n,n+1};V^{p(k)})}^{2} \right. \nonumber \\
	& \left. + C_3\left\| \frac{du^{n-j+1}(t)}{dt}\right\|_{L^2(I_{n-j,n-j+1};H)}^{2} + C_4\tau^{k} \left\| \frac{du^{n-j+1}(t)}{dt}\right\|_{L^2(I_{n-j,n-j+1};V^{p(k)})}^{2} \right]. \label{eq:interp4}
	\end{align}	
%}	
 	Unify the expression in \eqref{eq:interp1}--\eqref{eq:interp2} with the convention of $C_0^*=1$ and combine \eqref{eq:interp3}--\eqref{eq:interp4}  with  \eqref{eq:energy stability2}, it yields
	\begin{align} \label{eq:energy stability3}
			&	\left\| 	\frac{du^{n+1}(t)}{dt} \right\|_{L^2(I_{n,n+1};H)}^2 + A\tau^k \left\| 	\frac{du^{n+1}(t)}{dt}\right\|_{L^2(I_{n,n+1};V^{p(k)})}^2 + E(u^{n+1}) - E(u^{n}) \nonumber \\
			\le & C_L \left(C_3+C_1\sum_{j=0}^{k-1} C_j^* \right) \left\| \frac{du^{n+1}(t)}{dt}\right\|_{L^2(I_{n,n+1};H)}^{2}  + C_L\left(C_4+C_2\sum_{j=0}^{k-1} C_j^* \right)\tau^k \left\|\frac{du^{n+1}(t)}{dt}\right\|_{L^2(I_{n,n+1};V^{p(k)})}^{2} \nonumber \\
			& + C_LC_3\sum_{j=1}^{k-1}  C_j^* \left\| \frac{du^{n-j+1}(t)}{dt}\right\|_{L^2(I_{n-j,n-j+1};H)}^{2}  + C_LC_4 \tau^k\sum_{j=1}^{k-1}  C_j^*  \left\| \frac{du^{n-j+1}(t)}{dt}\right\|_{L^2(I_{n-j,n-j+1};V^{p(k)})}^{2}.
	\end{align}
	%{
	Adding  $C_LC_3  \sum_{j=1}^{k-2} \bm{\overline{C}}_{j+1}\left\| \frac{du^{n-j+1}(t)}{dt}\right\|_{L^2(I_{n-j,n-j+1};H)}^{2}$, $C_LC_4 \sum_{j=1}^{k-2} \bm{\overline{C}}_{j+1} \tau^k \left\| \frac{du^{n-j+1}(t)}{dt}\right\|_{L^2(I_{n-j,n-j+1};V^{p(k)})}^{2}$ to both sides of \eqref{eq:energy stability3} and utilizing \eqref{eq:cjbar-rel}, we deduce
	\begin{align} \label{eq:energy stability5}
	&E(u^{n+1})+	\left( 1 - C_L \left(C_3 +C_1\sum_{j=0}^{k-1} C_j^*  \right) \right)\left\| 	\frac{du^{n+1}(t)}{dt} \right\|_{L^2(I_{n,n+1};H)}^2 \nonumber \\
	& + \left( A - C_L\left(C_4 +C_2\sum_{j=0}^{k-1} C_j^* \right) \right)\tau^k \left\| 	\frac{du^{n+1}(t)}{dt}\right\|_{L^2(I_{n,n+1};V^{p(k)})}^2  \nonumber  \\
	& + C_LC_3  \sum_{j=1}^{k-2} \bm{\overline{C}}_{j+1}\left\| \frac{du^{n-j+1}(t)}{dt}\right\|_{L^2(I_{n-j,n-j+1};H)}^{2} +  C_LC_4 \sum_{j=1}^{k-2} \bm{\overline{C}}_{j+1} \tau^k \left\| \frac{du^{n-j+1}(t)}{dt}\right\|_{L^2(I_{n-j,n-j+1};V^{p(k)})}^{2}  \nonumber \\
	\le &  E(u^{n}) +  C_LC_3\bm{\overline{C}}_{1} \left\| \frac{du^{n}(t)}{dt}\right\|_{L^2(I_{n-1,n};H)}^{2}  + C_LC_4 \bm{\overline{C}}_{1}  \tau^k \left\| \frac{du^{n}(t)}{dt}\right\|_{L^2(I_{n-1,n};V^{p(k)})}^{2}  \nonumber \\
	& + C_LC_3 \sum_{j=2}^{k-1} \bm{\overline{C}}_j \left\| \frac{du^{n-j+1}(t)}{dt}\right\|_{L^2(I_{n-j,n-j+1};H)}^{2} + C_L C_4 \sum_{j=2}^{k-1} \bm{\overline{C}}_j \tau^k \left\| \frac{du^{n-j+1}(t)}{dt}\right\|_{L^2(I_{n-j,n-j+1};V^{p(k)})}^{2} .
	\end{align}
	Then the modified energy decaying property \eqref{eq:energy stable}
	follows from \eqref{eq:constant choose1}--\eqref{eq:constant choose2} and \eqref{eq:energy stability5}.
	
\end{proof}

\begin{Remark}
 The requirement postulated in \eqref{eq:constant choose1}--\eqref{eq:constant choose2} is sufficient but not necessary. It would be interesting to find the optimal choice of $A$ that guarantees energy stability.
Our numerical results presented in the next section suggest that the stability is insensitive to the choice of $A$ for a certain range well below the theoretical requirement that we have derived here. The theoretically optimal $p(k)$ comes from two consideration: (1) it should be as small as possible to avoid large artificial error; (2) it should be large enough to control the nonlinear term when combined with the original dissipation term. The critical  $p(k)$ is derived  under the condition that we wish to take $A$ to be independent of $\epsilon$. Larger $p(k)$ may lead to  larger error, especially for high-frequency solutions. We can use a relatively low-order regularization to reduce the artificial errors, while this treatment may give us an $A$ that depends on $\epsilon$ or $\tau$. In \cite{cheng2019third},   a $A\tau^{2}\left(\Delta^2 u^{n+1} - \Delta^2 u^n\right)$ regularization term is added in their third-order ETD scheme, the exponent $p(3)=1$ is smaller than our theoretical optimal value, thus the artificial error is supposed to be smaller, while their stabilized coefficient is of order $\mathcal{O}(\epsilon^{-2})$. This is due to the fact  that the regularization combined with original dissipation term is insufficient to control the explicit nonlinear term since temporal accuracy of the artificial regularization term has to be  kept, thus the  surface diffusion term is involved in the energy stability analysis which is of order $\mathcal{O}(\epsilon^2)$.
\end{Remark}

\begin{Remark}
	Note that the regularization term $A\tau^k\frac{\partial }{\partial t}\cL^{p(k)} u$ can be replaced by a Dupont-Douglas type $A\tau^{k-1}\cL^{p(k)}\left( u^{n+1} - u^n\right)$ as long as the exact solution is smooth enough. The continuous form adopted here is  consistent with the spirit of $ETD$.  As we mentioned before, the introduction of high-order regularization may lead to large truncation error for the high-frequency solution, while the convergence in finite time is preserved as we can see in the following numerical tests. For long-time coarsening process with random initial data, two alternative treatment could be considered in the initial evolution until a relatively smooth solution is obtained: 1. use some high-order methods without stabilization; 2. use variable time step-size in the simulation, which is supposed to be very small in the beginning in order to handle the effect of high-frequency components of solution.
	
\end{Remark}

%%%%%	
	\section{Numerical example}
	In this section, we applied the abstract framework to the NSS equation with $k=4$.  The two dimensional domain $\Omega=[0,L]^2$ with periodic boundary condition is considered. In this case, $\cL = \Delta^2$, $\cN(u) = -\nabla\cdot \left( \frac{\nabla u}{1+|\nabla u|^2}\right)$ in our general framework. %{
	Thus the abstract functional spaces are specified to $H=\{f: f\in L^2 ~\text{with zero mean}\}$, $V^{1/2}=\{f: f\in H_{per}^1~\text{with zero mean}\}$, $V^\alpha=\{f: f\in H_{per}^{2\alpha}~\text{with zero mean}\}$
	%}
	and the Lipschitz continuity takes the following form
	\begin{align} \label{eq:LIP nss}
	\left\| \cN(u) - \cN(v) \right\|_{V^{-\frac12}} \le  \left\| u-v\right\|_{V^\frac12},
	\end{align}
	that is $\beta=\gamma=\frac12$, then equation \eqref{eq:p(k)2} indicates $p(k)=k/2$.
	
	Here we use a fourth-order ETD-MS method to approximate equation \eqref{eq:nss}, i.e.,
	\begin{align} \label{eq:etdms nss}
	\frac{du^{n+1}(t)}{dt} + \epsilon \Delta^2u^{n+1}(t) + A\tau^4 \frac{d}{dt} \Delta^4 u^{n+1}(t) = \sum_{i=0}^{3} \ell_i(t-t_n) \cN(u^{n-i}), \quad  t \in [t_n, t_{n+1}].
	\end{align}
	The Lagrange basis functions $\{\ell_i(s)\}_{i=0}^{3}, ~0\le s \le \tau$ are
	\begin{align*}
		\ell_0(s) &= \frac{s^3}{6\tau^3} + \frac{s^2}{\tau^2} + \frac{11s}{6\tau} + 1, \quad 	\ell_1(s) = -\frac{s^3}{2\tau^3} - \frac{5s^2}{2\tau^2} - \frac{3s}{\tau}, \\
		\ell_2(s) & = \frac{s^3}{2\tau^3} + \frac{2s^2}{\tau^2} + \frac{3s}{2\tau}, \quad 	\ell_3(s) = -\frac{s^3}{6\tau^3} - \frac{s^2}{2\tau^2} - \frac{s}{3\tau}.
	\end{align*}
	The corresponding constants in \eqref{eq:ell L2 integral} are $C_1^*=\sqrt{9143/3780}, \ C_2^*=\sqrt{157441/7560}, \ C_3^*=\sqrt{212/945}$ and 
	\begin{align} \label{eq: barC}
		\bm{\overline{C}}_1 = & C_1^*+ C_2^* +C_3^*  \nonumber \\
		= & \left(\sqrt{18286}+\sqrt{157441}+\sqrt{1696}\right) / \sqrt{7560}.
	\end{align}
	The constants in \eqref{eq:bad term} are $C_1=\frac38\hat{C}^{4/3}, \ C_2=\frac18\hat{C}^{-4}, \ C_3=\frac38\tilde{C}^{4/3}, \ C_4=\frac18\tilde{C}^{-4}$. Substitute these into \eqref{eq:constant choose1}--\eqref{eq:constant choose2}, then we can take $\hat{C}=\tilde{C}=\left(\frac{4}{3(1+\bm{\overline{C}}_1)}\right)^{3/4}$ to satisfy
	\begin{align*}
		1 - \frac{3\hat{C}^{4/3}}{8}\left(2+\bm{\overline{C}}_1\right) \ge \frac{3\hat{C}^{4/3}}{8}\bm{\overline{C}}_1.
	\end{align*}
	Thus the regularization coefficient $A= \frac{\hat{C}^{-4}}{8}\left(1+\bm{\overline{C}}_1\right)=\frac{27\left(1+\bm{\overline{C}}_1\right)^4}{512} $ with $\bm{\overline{C}}_1$ given in \eqref{eq: barC} is required. However, we will show the numerical results do not depend much on this relatively stringent condition in the following.
	
	 Note that initial steps are needed to start the high-order multi-step scheme \eqref{eq:etdms nss} and here we use the ETD-RK method to compute the first three numerical solutions. The general expression of the fourth-order ETD-RK method is illustrated in  \cite{kassam2005fourth-order} as
	\begin{equation} \label{eq:ETDRK4}
	\left\{
	\begin{aligned}
	a^n =& e^{-\epsilon\cL \tau/2}u^n + (-\epsilon\cL)^{-1}\left(e^{-\epsilon\cL \tau/2}-I\right)\cN(u^n,t^n),  \\
	b^n = & e^{-\epsilon\cL \tau/2}u^n + (-\epsilon\cL)^{-1}\left(e^{-\epsilon\cL \tau/2}-I\right)\cN(u^n,t^n+\tau/2),  \\
	c^n = & e^{-\epsilon\cL \tau/2}a^n + (-\epsilon\cL)^{-1}\left(e^{-\epsilon\cL \tau/2}-I\right) \left(2\cN(b^n,t^n+\tau/2)- \cN(u^n,t^n) \right),  \\
	u^{n+1}=& e^{-\epsilon\cL\tau} u^{n}+\tau^{-2} (-\epsilon\cL)^{-3}\left\{\left[-4+\epsilon\cL \tau +e^{-\epsilon\cL\tau}\left(4+3\epsilon \cL\tau+(-\epsilon\cL\tau)^{2}\right)\right] \cN\left(u^{n}, t^{n}\right)\right.  \\
	&+2\left[2-\epsilon\cL \tau-e^{-\epsilon\cL \tau}(2+\epsilon\cL\tau)\right]\left(\cN\left(a^{n}, t^{n}+\tau / 2\right)+\cN\left(b_{n}, t_{n}+\tau / 2\right)\right)   \\
	&\left.+\left[-4+3\epsilon \cL\tau-(-\epsilon\cL\tau)^{2}+e^{-\epsilon\cL\tau}(4+\epsilon\cL\tau)\right] \cN\left(c^{n}, t^{n}+\tau\right)\right\} .
	\end{aligned}
	\right.
	\end{equation}
	To solve  system \eqref{eq:etdms nss}, the spatial discretization is performed by the Fourier pseudo-spectral method with a resolution $N=128$ and this can be efficiently implemented via the fast Fourier transform. The convergence  and energy stability tests are provided.
	
	\subsection{Temporal convergence test}
	In this subsection, the fourth order temporal convergence of scheme \eqref{eq:etdms nss} is verified. The parameters are chosen as $L=2\pi$, $\epsilon=0.01$ and terminal time $T=1$. To test the convergence, an artificial forcing term $g$ is added in the right hand side of \eqref{eq:etdms nss} to make the exact solution $u(t)=e^{-t}\cos(2x)\cos(2y)$:
	\begin{align*}
	g &= (-1+64\varepsilon^2)u-\frac{8u}{1+2e^{-2t}[1-\cos(4 x)\cos(4 y)]} \\
	&\quad +\frac{16e^{-2t}u}{[1+2e^{-2t}(1-\cos(4x)\cos(4y))]^2} [\cos(4x) + \cos(4y) - 2\cos(4x)\cos(4y)],
	\end{align*}
	and the discrete $L^2$ error is calculated.
	Firstly, the effect of regularized coefficient $A$ on  the convergence is examined in Table~\ref{tab: temporal convergence}. From which we can see, the errors grow linearly with the increasing of $A$ and the convergence order is preserved for all of the choice of $A$. Then let $A$ be fixed, results for ${p(k)}=1.5$, ${p(k)}=1.9$, ${p(k)}=2.1$ and ${p(k)}=2.5$ are presented in Table~\ref{tab: temporal convergence2}. While higher order regularization leads to larger error, clear fourth order convergence rates have been observed for all of ${p(k)}$.
		\begin{table}[H]
		\centering
		\caption{Temporal convergence of \eqref{eq:etdms nss} with  ${p(k)}=2$.}\label{tab: temporal convergence}
		\begin{tabular}{|c|c|c|c|c|c|c|c|c|}
			\hline
			\multirow{2}{*}{$\tau$} & \multicolumn{2}{c|}{$A=1$}  & \multicolumn{2}{c|}{$A=5$ } & \multicolumn{2}{c|}{$A=10$ } & \multicolumn{2}{c|}{$A=\frac{27\left(1+\bm{\overline{C}}_1\right)^4}{512}$ }  \\
			\cline{2-9}
			& error & order  & error & order  & error & order & error & order \\
			\hline
			2.50E-03 & 6.53e-07 &   & 3.26e-06 &  & 6.52e-06 &  & 1.12e-04 & \\
			\hline
			1.25E-03 & 4.10e-08 & 3.993  & 2.05e-07 & 3.992 & 4.10e-07 &  3.991 & 7.18e-06 & 3.959\\
			\hline
			6.25E-04 & 2.57e-09 & 3.997  & 1.28e-08 & 3.996 & 2.57e-08 & 3.996 & 4.50e-07 & 3.994  \\
			\hline
			3.13E-04 & 1.59e-10 & 4.009  & 8.04e-10 & 3.999 & 1.61e-09 & 4.000 & 2.82e-08 & 3.998 \\
			\hline
			1.56E-04 & 8.94e-12 & 4.157  & 4.96e-11 & 4.017 & 9.99e-11 & 4.007 & 1.76e-09 & 4.000  \\
			\hline
		\end{tabular}
	\end{table}

	\begin{table}[H]
	\centering
	\caption{Temporal convergence of \eqref{eq:etdms nss} with $ A=\frac{27\left(1+\bm{\overline{C}}_1\right)^4}{512}$.}\label{tab: temporal convergence2}
	\begin{tabular}{|c|c|c|c|c|c|c|c|c|}
		\hline
		\multirow{2}{*}{$\tau$} & \multicolumn{2}{c|}{${p(k)}=1.5$} & \multicolumn{2}{c|}{${p(k)}=1.9$} & \multicolumn{2}{c|}{${p(k)}=2.1$ } & \multicolumn{2}{c|}{${p(k)}=2.5$}  \\
		\cline{2-9}
		& error & order & error & order & error & order & error & order  \\
		\hline
		2.50E-03  & 1.43e-05  & & 6.03e-05 &   & 1.34e-04 & &  7.14e-04 &  \\
		\hline
		1.25E-03 & 8.99e-07 &3.992 & 3.83e-06 & 3.976  & 8.79e-06 & 3.934 & 5.29e-05  & 3.755 \\
		\hline
		6.25E-04 & 5.63e-08 & 3.996 & 2.40e-07 & 3.995  & 5.52e-07 & 3.993  & 3.56e-06 & 3.891  \\
		\hline
		3.13E-04 & 3.52e-09 & 3.998& 1.50e-08 & 3.998  & 3.46e-08 & 3.998  & 2.25e-07 & 3.983\\
		\hline
		1.56E-04 & 2.19e-10 & 4.006 & 9.41e-10 & 3.998   & 2.16e-09 & 4.000  & 1.41e-08  & 3.998 \\
		\hline
	\end{tabular}
\end{table}

 	\subsection{Simulation of coarsening process}
 	In this subsection, the physically interesting coarsening process is simulated.  	The parameters are now set as $L=12.8$, $\epsilon=0.005$,  $T=30000$ and $\tau=10^{-3}$. Two choices for the regularization coefficient  $A=\frac{27\left(1+\bm{\overline{C}}_1\right)^4}{512}$ and $A=10$ are tested. The scaling laws for the energy $E$, average surface roughness $h$ and the average slope $m$ will be shown, that is $E \sim O(-\ln (t))$, $h \sim O(t^{\frac{1}{2}})$ and  $m\sim O(t^{\frac{1}{4}})$ as $t\rightarrow\infty$. (See \cite{golubovic1997interfacial, li2003thin, li2004epitaxial} and references therein). The corresponding definitions for these physical quantities are
 	\begin{align*}
 	E(u) & =\left(-\frac{1}{2}\ln(1+|\nabla u|^2),1\right) + \frac{\varepsilon^2}{2}\|\Delta  u\|^2, \\
 	h(u,t) &= \sqrt{\frac{h^2}{|\Omega|}\sum_{x} |u(x,t)-\bar{u}(t)|^2}, \quad \mbox{with} \quad \bar{u}(t):=\frac{h^2}{|\Omega|}\sum_{x} u(x,t), \\
 	m(u,t) &= \sqrt{\frac{h^2}{|\Omega|}\sum_{x} |\nabla u(\textbf{x}_{i,j},t)|^2}.
 	\end{align*}

 	The snapshots of the numerical solution \eqref{eq:etdms nss} at time $t = $ 1, {5000, 10000, 15000, 20000, 30000} with $A=\frac{27\left(1+\bm{\overline{C}}_1\right)^4}{512}$ and $A=10$ are displayed in Figures~\ref{fig: ss} and \ref{fig: ss2} respectively. The evolution of $E$, $h$ and $m$ with  two choices of $A$ is established in Figures~\ref{fig: energy}--\ref{fig: slope}, and the linear fitting results for the solution of \eqref{eq:etdms nss} in time interval $[1, 400]$ are also presented, which is consistent with the theoretical scaling laws.
 	
 	Obvious differences of numerical solutions are observed  for two choices of $A$ with a uniform step-size $\tau=10^{-3}$, in Figures \ref{fig: ss}--\ref{fig: slope}. Since random initial data is used, the artificial error arising from high-order regularization term may be large in the initial stage, resulting in the sensitivity to stabilized coefficient $A$. Therefore, an additional coarsening simulation is performed in the following with a variable time step-size, which is set as: $\tau=10^{-6}$ for $t\le 1$, $\tau=10^{-5}$ for $1<t\le 10$, $\tau=10^{-4} $ for $10<t\le 100$ and $\tau=10^{-3}$ for $t>100$ and other parameters are kept. The related results are displayed in Figures~\ref{fig: ss_1}--\ref{fig: ss_2} and Figures~\ref{fig: energy1}--\ref{fig: slope1}. From which we can see, the sensitivity to $A$ is significantly reduced. Also, minor differences of solutions between two choices of $A$ are observed in Figure~\ref{fig: re}.

 	\begin{Remark}
 		Thanks to the anonymous reviewer for pointing out the  differences between snapshots of numerical solution  for two choices of $A$ at the same time level, in Figures \ref{fig: ss}--\ref{fig: ss2}. In fact, the introduction of high-order regularization term may result in large truncation error in the beginning since random initial data is chosen, and further lead to different steady phase states. The same  simulation of coarsening process is performed with variable time step-sizes and other parameters kept. The time step-size is taken to be small to handle the effect of high-order regularization during the initial evolution of solution and a relatively large step-size is adopted after a smooth solution obtained. With this treatment being taken, the insensitivity to the stabilized coefficient $A$ of our numerical scheme can be observed in Figures~\ref{fig: ss_1}--\ref{fig: slope1}. This exactly shows the necessity of variable step-size in the long-time simulation of coarsening process. The choice of variable step-size is empirical here, and similar idea has been applied in \cite{li2018second,wang2020optimal}. There are also various adaptive methods regarding to the choice of time step-size, see \cite{chen2019second,feng2015,qiao2011,zhang2012}. In fact, one of our authors, Xiaoming Wang, has already pointed out the feasibility and necessity of using adaptive strategies or hybrid approach (utilize some alternative high-order methods without regularization for the initial stage) for long-time computation in our previous paper \cite{chen2020energy}, and it would be our future work to conduct the adaptive time-stepping strategy with a posterior estimate.
 	\end{Remark}

 		\begin{figure}[ht]
 		\centering
 		\noindent\makebox[\textwidth][c] {
 			\begin{minipage}{0.3\textwidth}
 				\includegraphics[width=\textwidth]{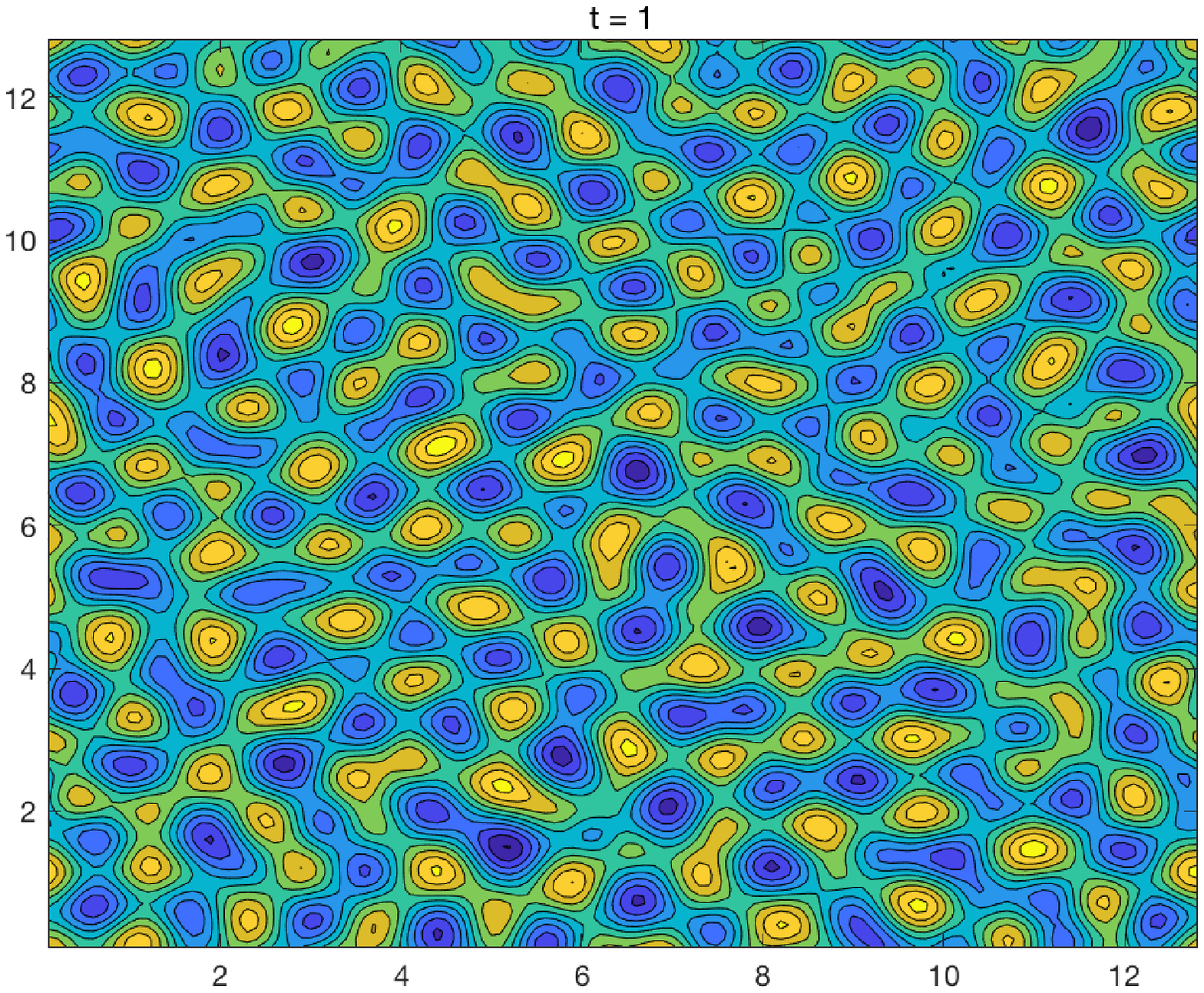}
 			\end{minipage}
 			\begin{minipage}{0.3\textwidth}
 				\includegraphics[width=\textwidth]{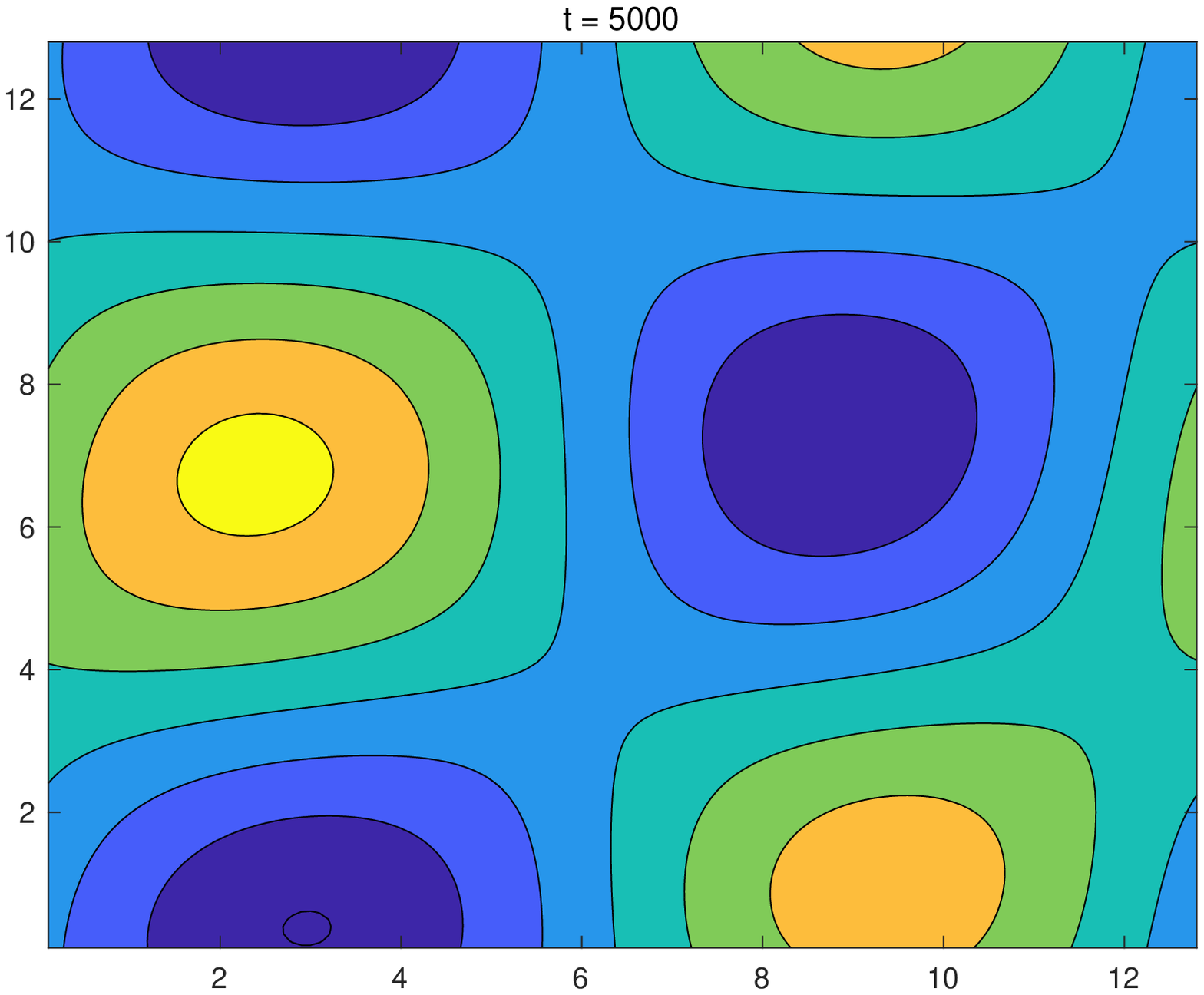}
 			\end{minipage}
 			\begin{minipage}{0.3\textwidth}
 				\includegraphics[width=\textwidth]{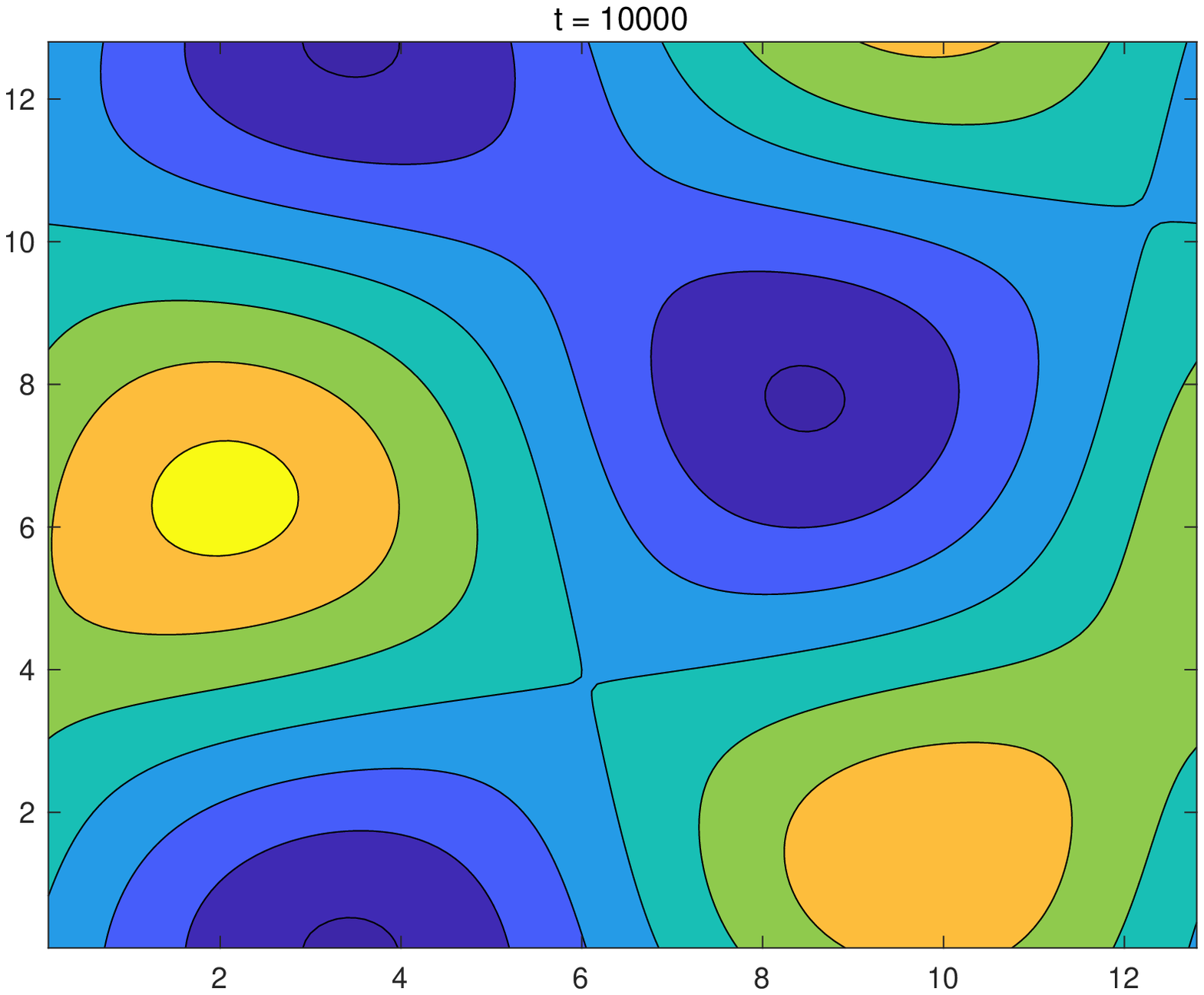}
 			\end{minipage}
 		}
 		\noindent\makebox[\textwidth][c] {
 			\begin{minipage}{0.3\textwidth}
 				\includegraphics[width=\textwidth]{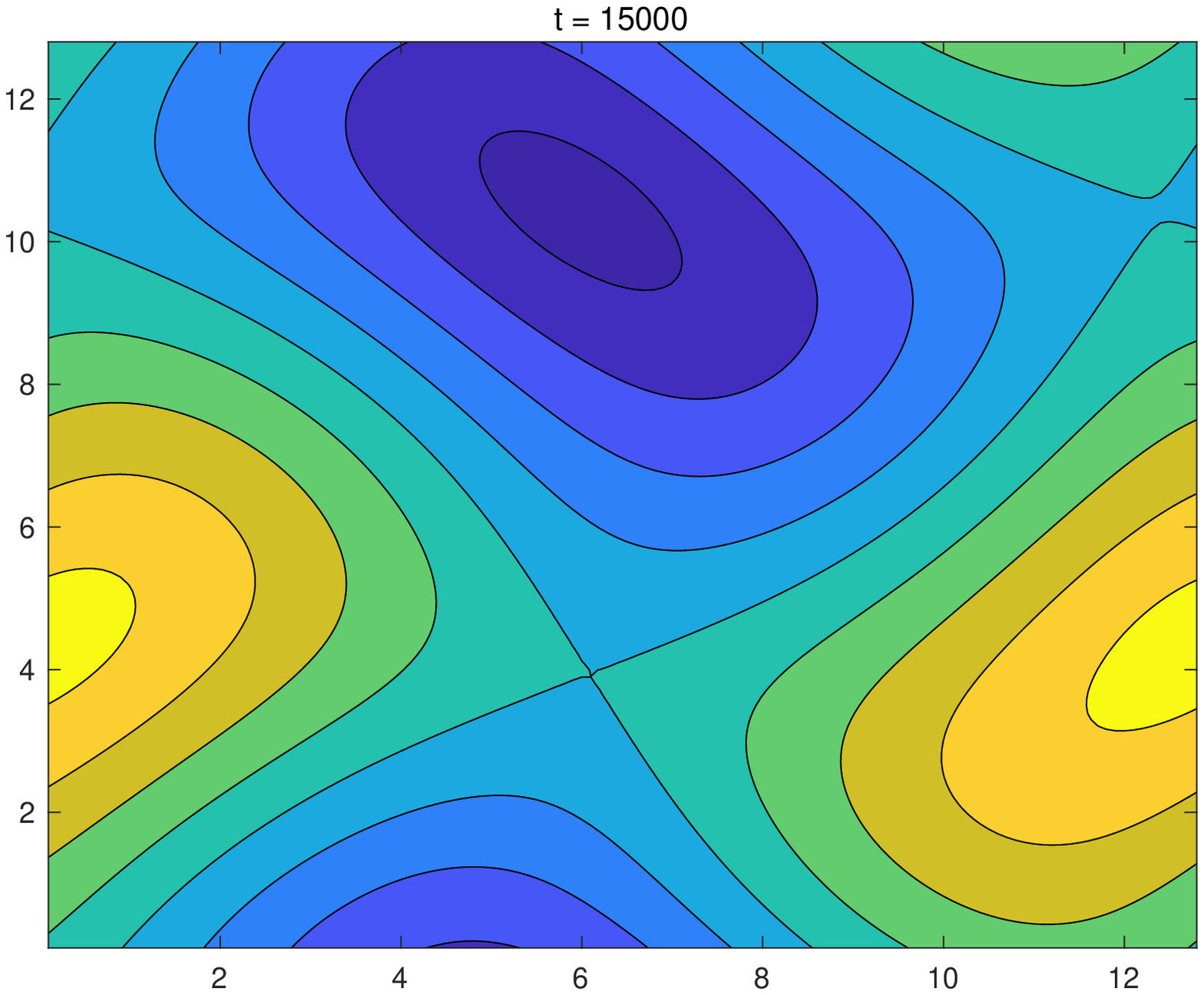}
 			\end{minipage}
 			\begin{minipage}{0.3\textwidth}
 				\includegraphics[width=\textwidth]{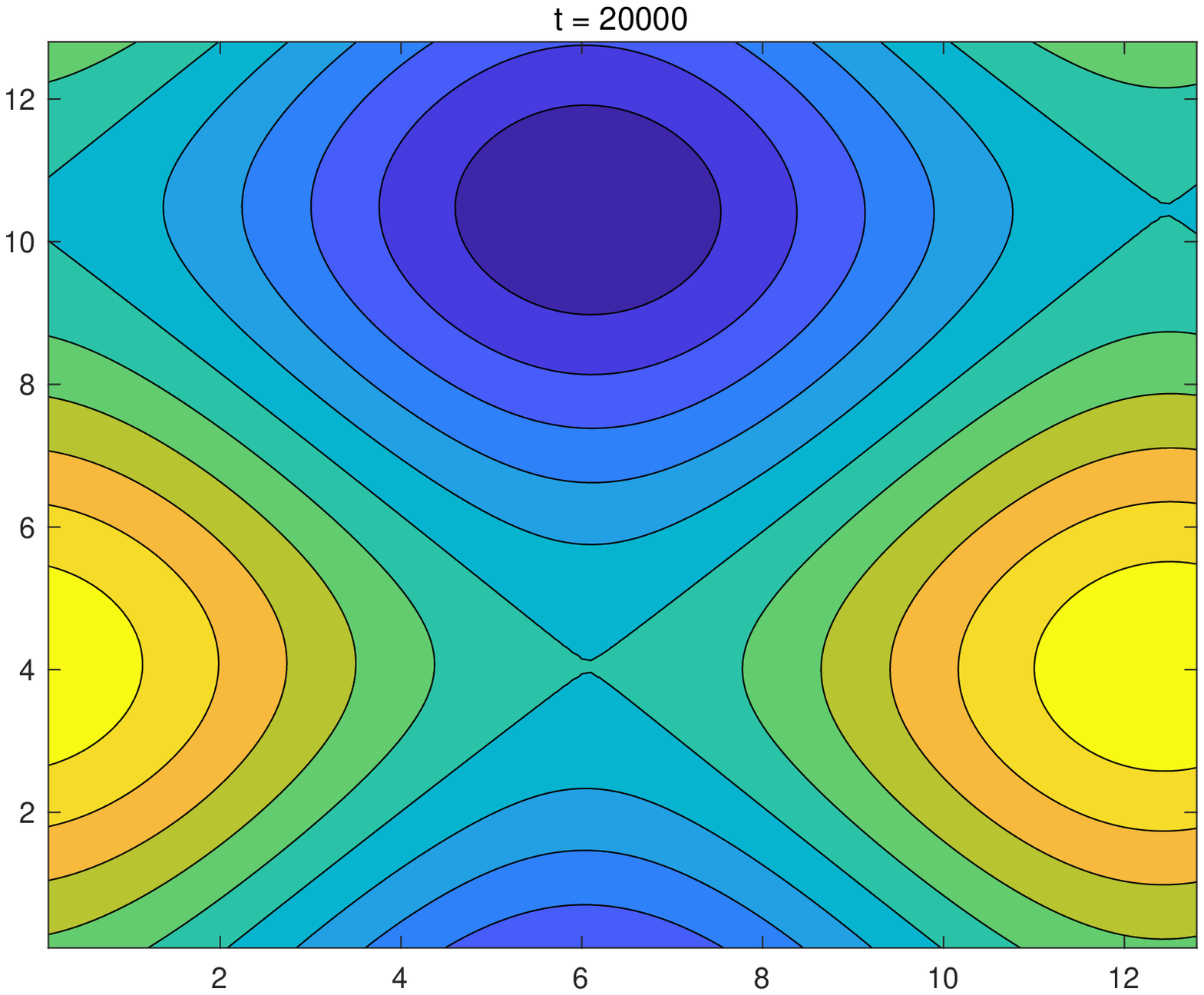}
 			\end{minipage}
 			\begin{minipage}{0.3\textwidth}
 				\includegraphics[width=\textwidth]{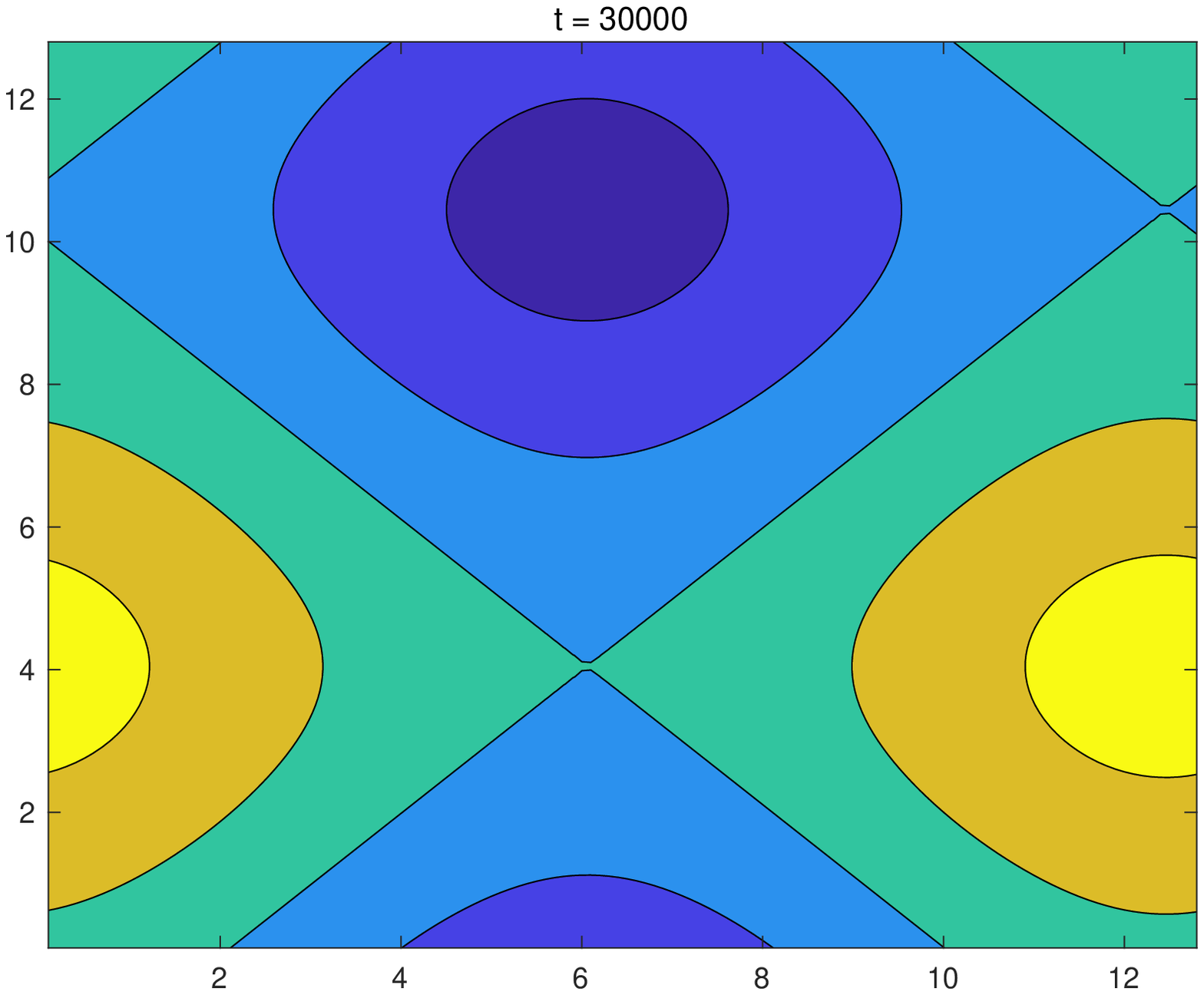}
 			\end{minipage}
 		}
 		\caption{Snapshots of the numerical solutions of scheme \eqref{eq:etdms nss} with $A=\frac{27\left(1+\bm{\overline{C}}_1\right)^4}{512}$, $\tau=10^{-3}$.}\label{fig: ss}
 	\end{figure}
 	
 	\begin{figure}[ht]
 		\centering
 		\noindent\makebox[\textwidth][c] {
 			\begin{minipage}{0.3\textwidth}
 				\includegraphics[width=\textwidth]{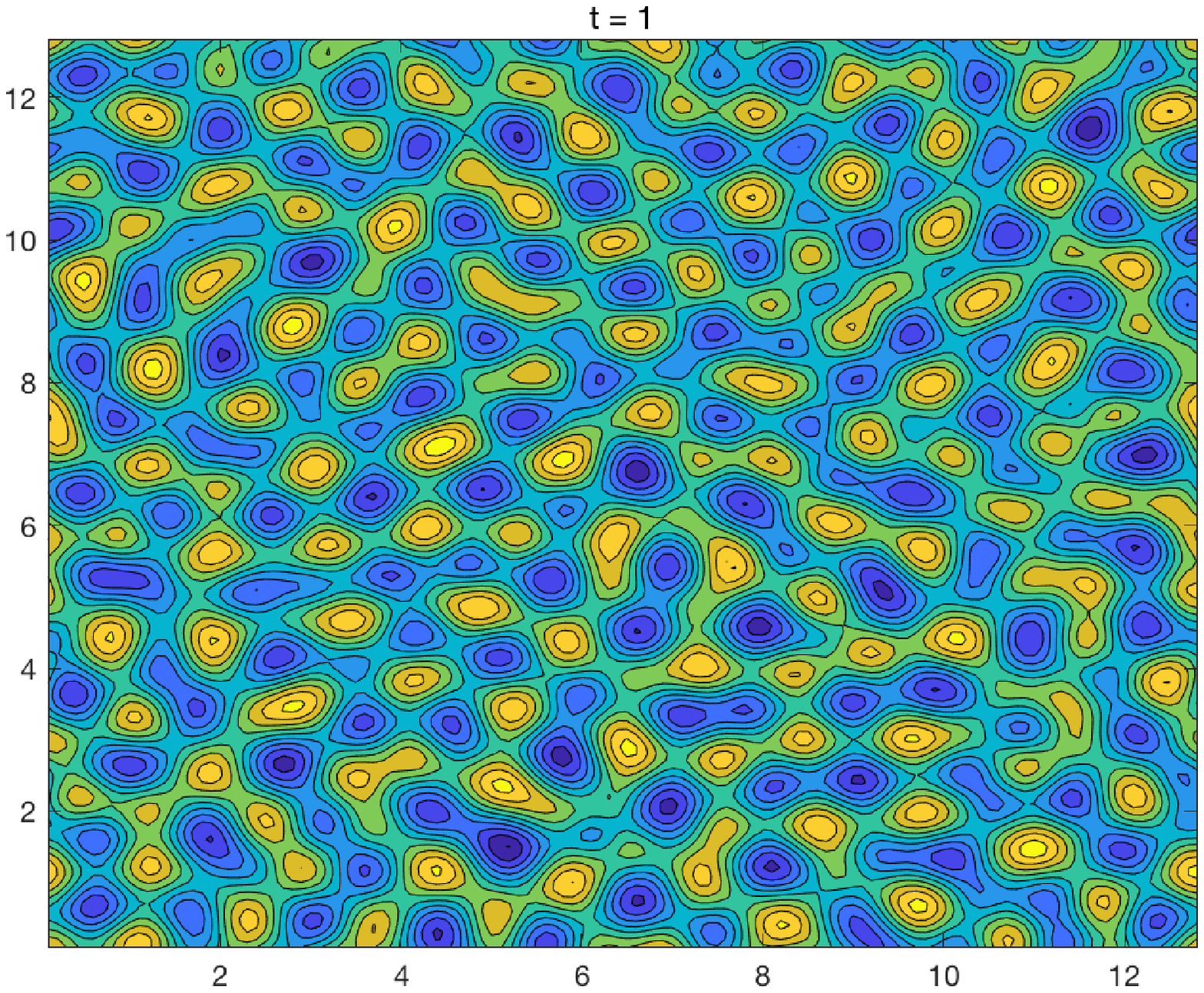}
 			\end{minipage}
 			\begin{minipage}{0.3\textwidth}
 				\includegraphics[width=\textwidth]{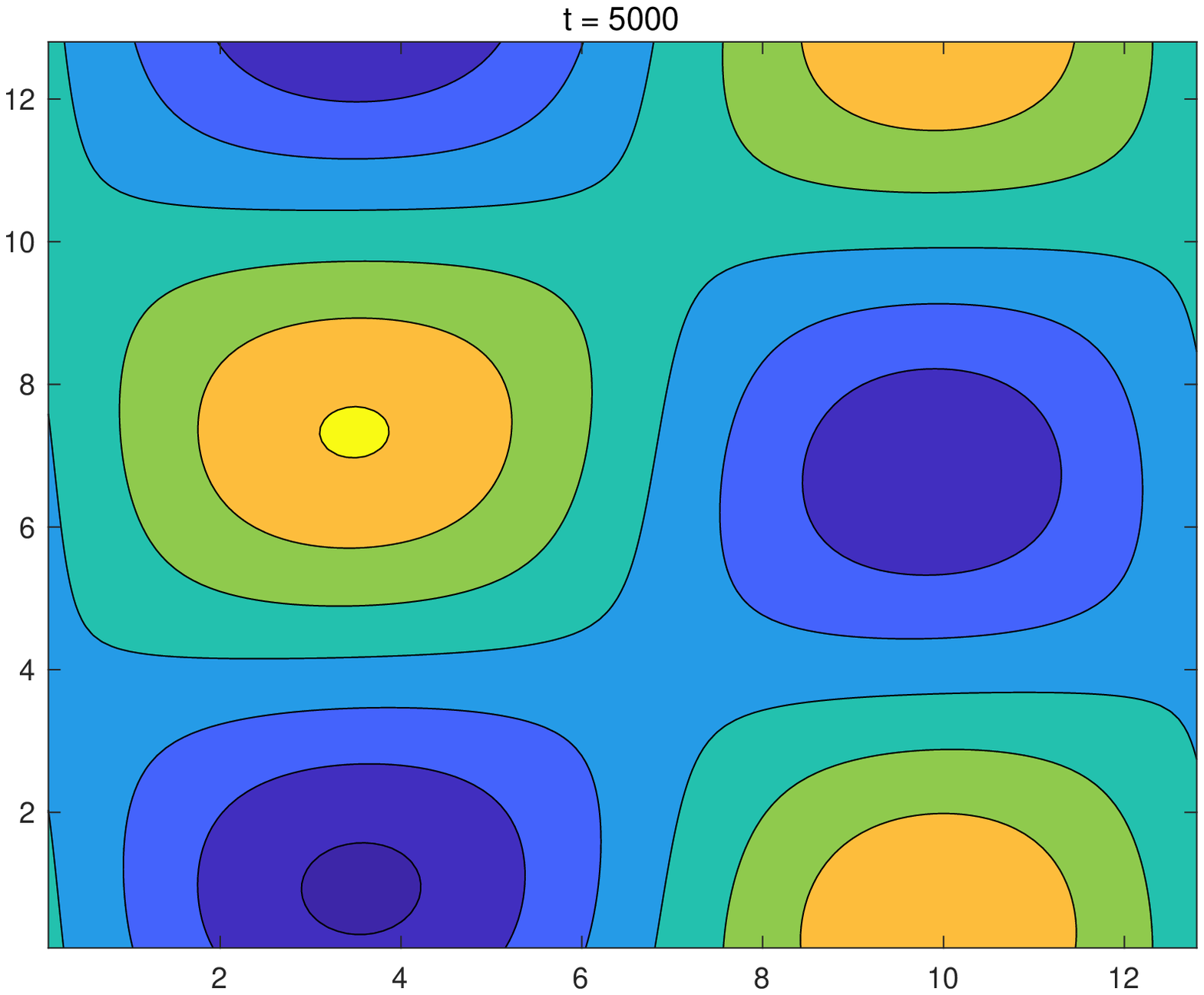}
 			\end{minipage}
 			\begin{minipage}{0.3\textwidth}
 				\includegraphics[width=\textwidth]{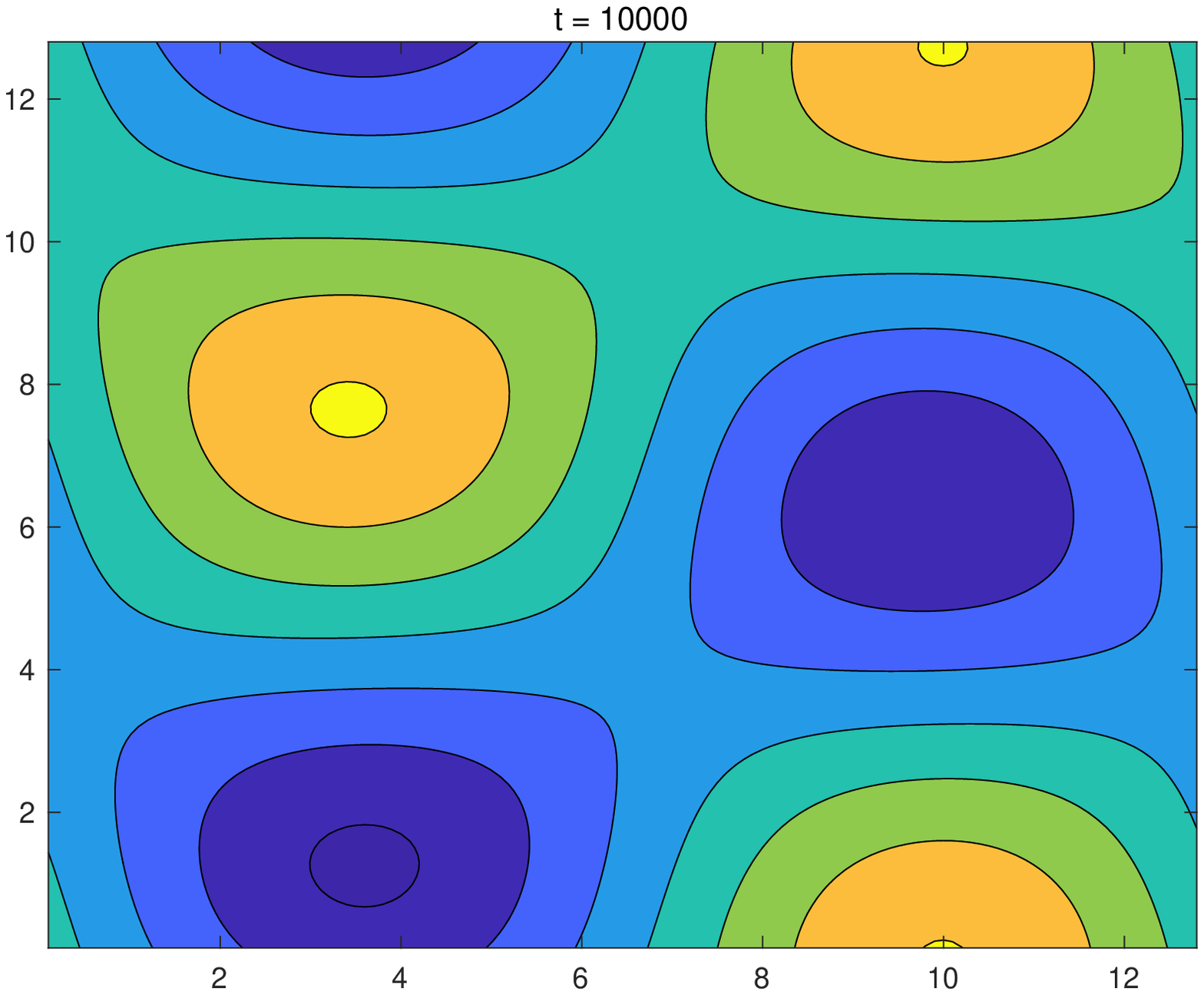}
 			\end{minipage}
 		}
 		\noindent\makebox[\textwidth][c] {
 			\begin{minipage}{0.3\textwidth}
 				\includegraphics[width=\textwidth]{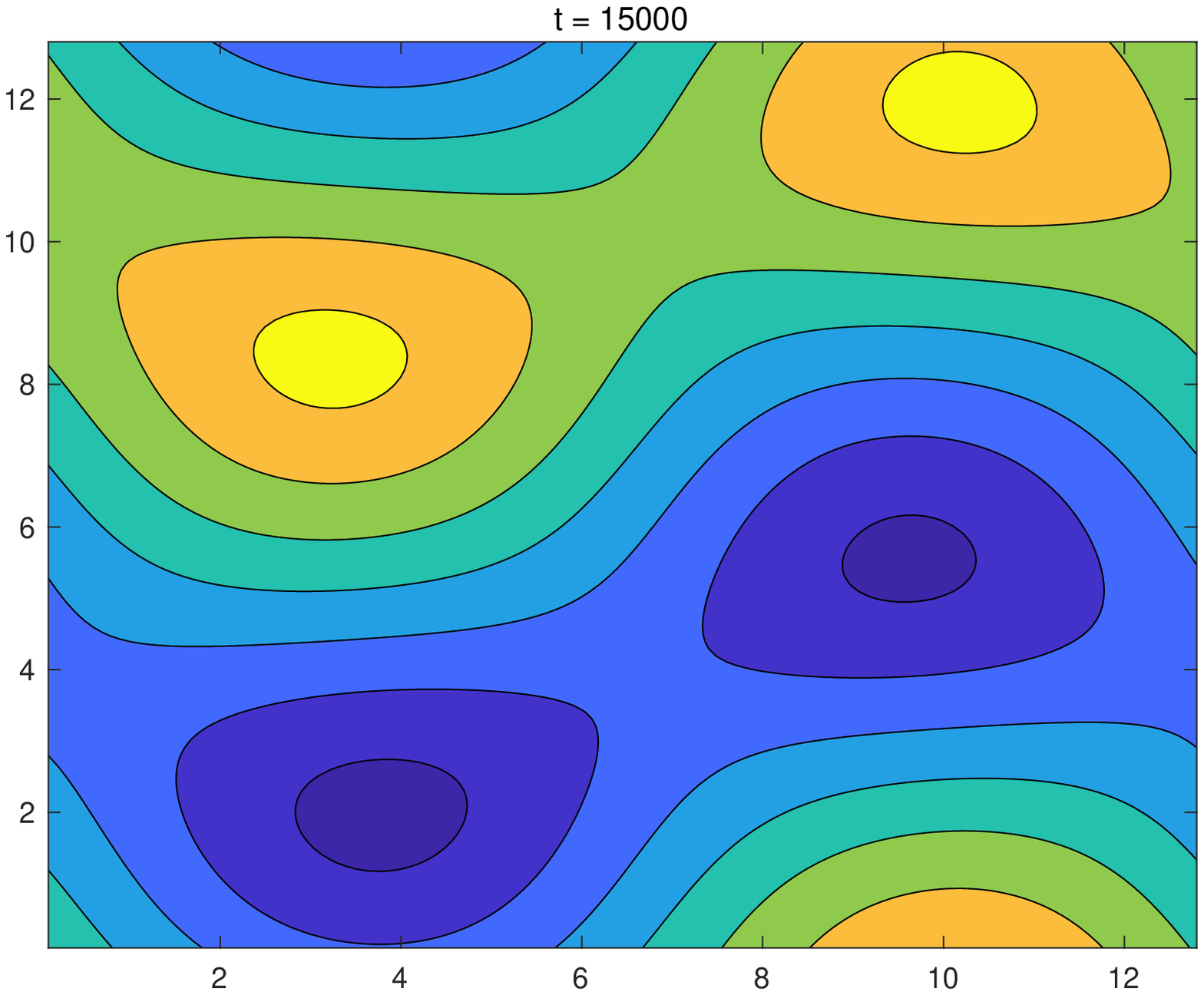}
 			\end{minipage}
 			\begin{minipage}{0.3\textwidth}
 				\includegraphics[width=\textwidth]{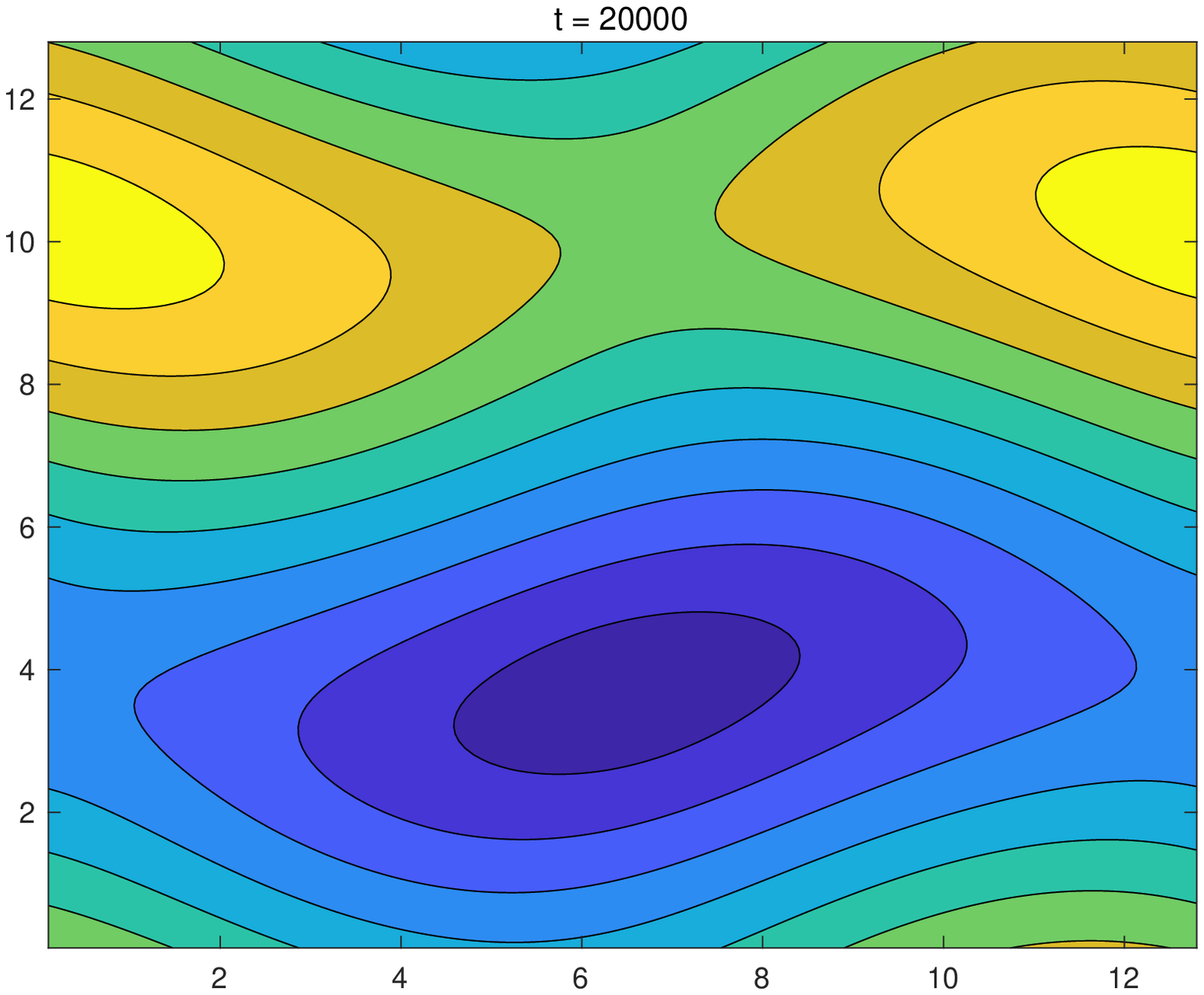}
 			\end{minipage}
 			\begin{minipage}{0.3\textwidth}
 				\includegraphics[width=\textwidth]{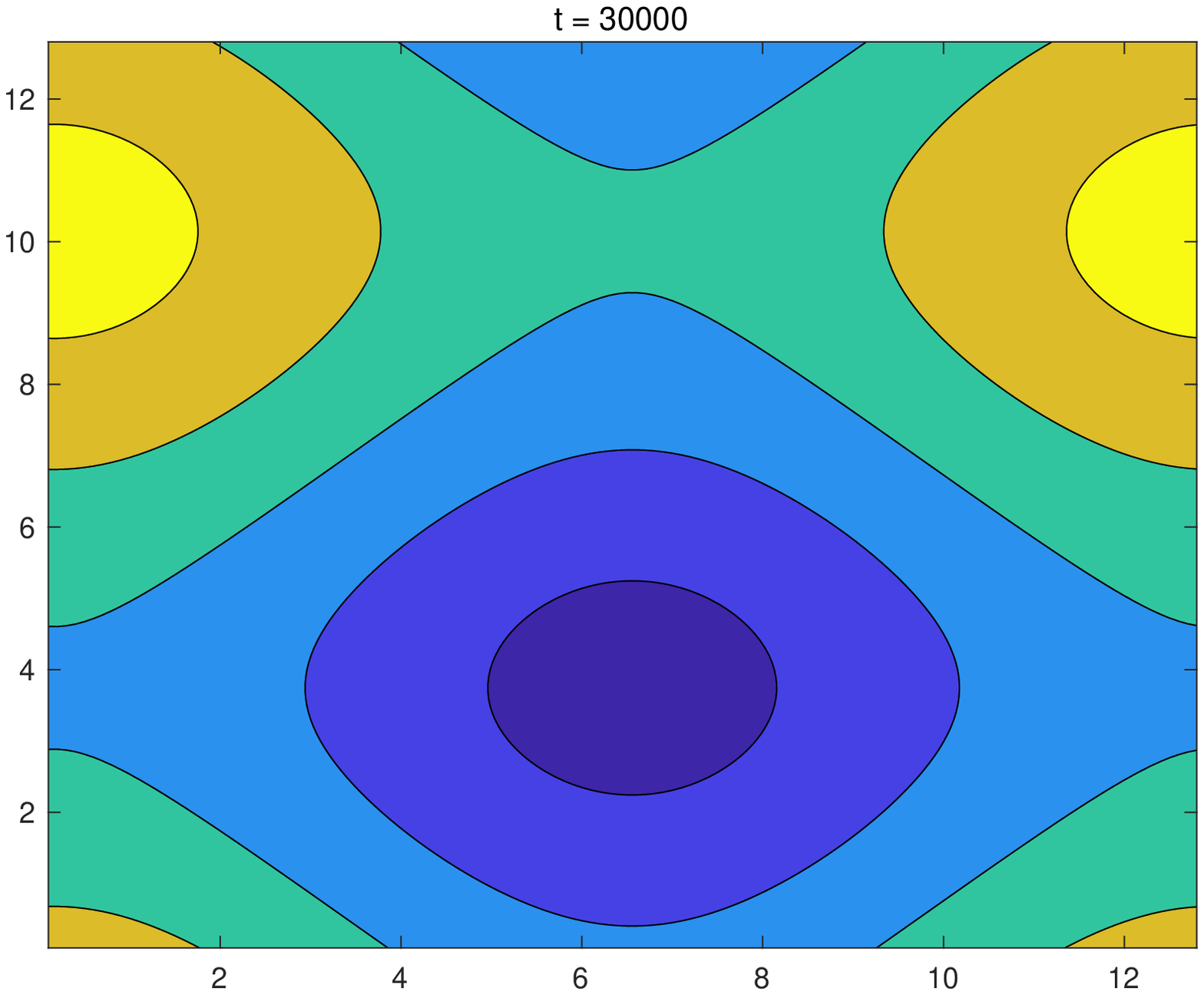}
 			\end{minipage}
 		}
 		\caption{Snapshots of the numerical solutions of scheme \eqref{eq:etdms nss} with $A=10$, $\tau=10^{-3}$.}\label{fig: ss2}
 	\end{figure}
 	
 	\begin{figure}[ht]
 		\begin{center}
 			\includegraphics[width =0.8\textwidth]{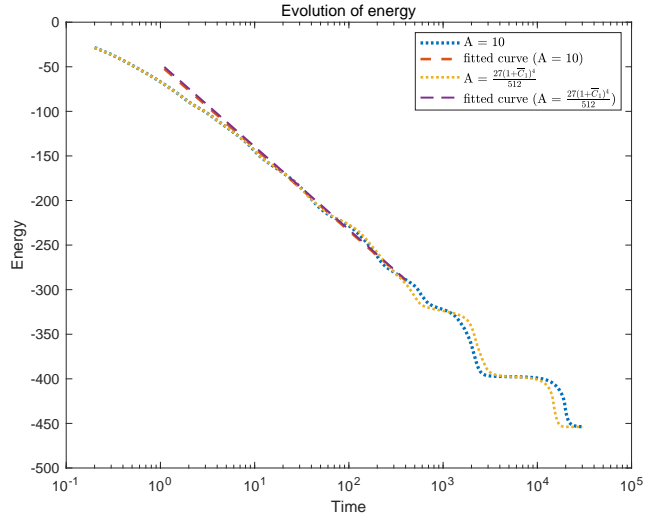}
 			\caption{Semi-log plot of the energy $E$ of scheme scheme \eqref{eq:etdms nss} with $\tau=10^{-3}$. Fitted line has the form $a \ln (t) + b$, with coefficients $a = -40.27$, $b = -48.5$ for $A=10$  and $a = -40.45$, $b = -46.35$ for $A= \frac{27\left(1+\bm{\overline{C}}_1\right)^4}{512}$.} \label{fig: energy}
 		\end{center}
 	\end{figure}
 	
 	\begin{figure}[ht]
 		\begin{center}
 			\includegraphics[width=0.8\textwidth]{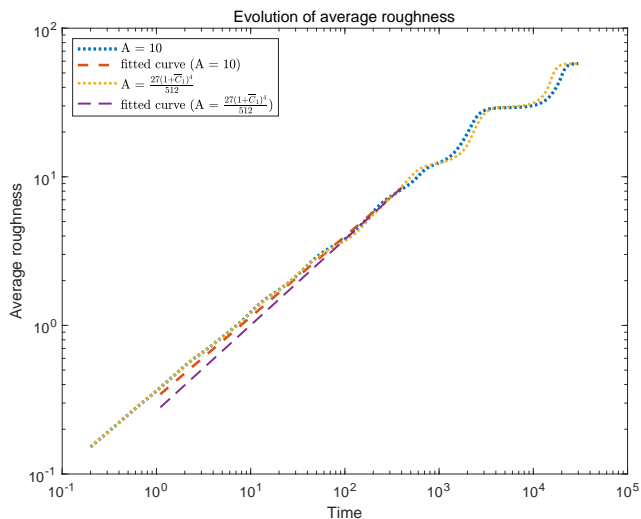}
 			\caption{The log-log plot of the average surface roughness $h$  of \eqref{eq:etdms nss} with  $\tau=10^{-3}$. Fitted lines have the form $a t^b$, with coefficients $a = 0.3269$, $b = 0.5423$ for $A=10$  and $a = 0.2657$, $b = 0.5771$ for $A= \frac{27\left(1+\bm{\overline{C}}_1\right)^4}{512}$.} \label{fig: roughness}
 		\end{center}
 	\end{figure}
 	
 	\begin{figure}[ht]
 		\begin{center}
 			\includegraphics[width=0.8\textwidth]{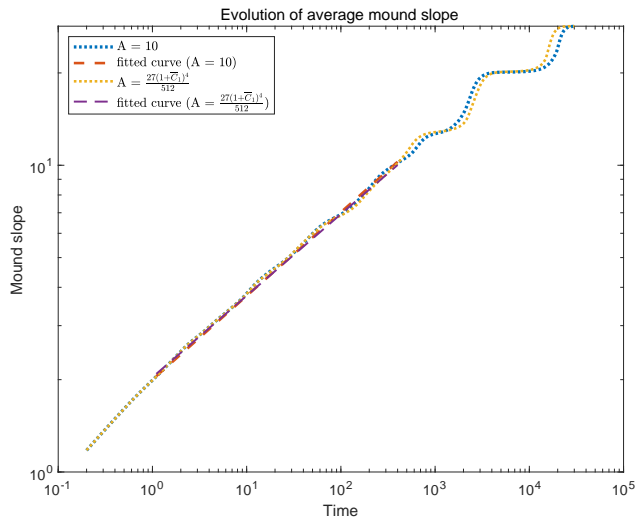}
 			\caption{The log-log plot of the average slope $m$ of \eqref{eq:etdms nss} with $\tau=10^{-3}$. Fitted lines have the form $a t^b$, with coefficients $a = 2$, $b = 0.2729$ for $A=10$  and $a = 2.036$, $b = 0.2659$ for $A= \frac{27\left(1+\bm{\overline{C}}_1\right)^4}{512}$.} \label{fig: slope}
 		\end{center}
 	\end{figure}

 	\begin{figure}[ht]
 		\centering
 		\noindent\makebox[\textwidth][c] {
 			\begin{minipage}{0.3\textwidth}
 				\includegraphics[width=\textwidth]{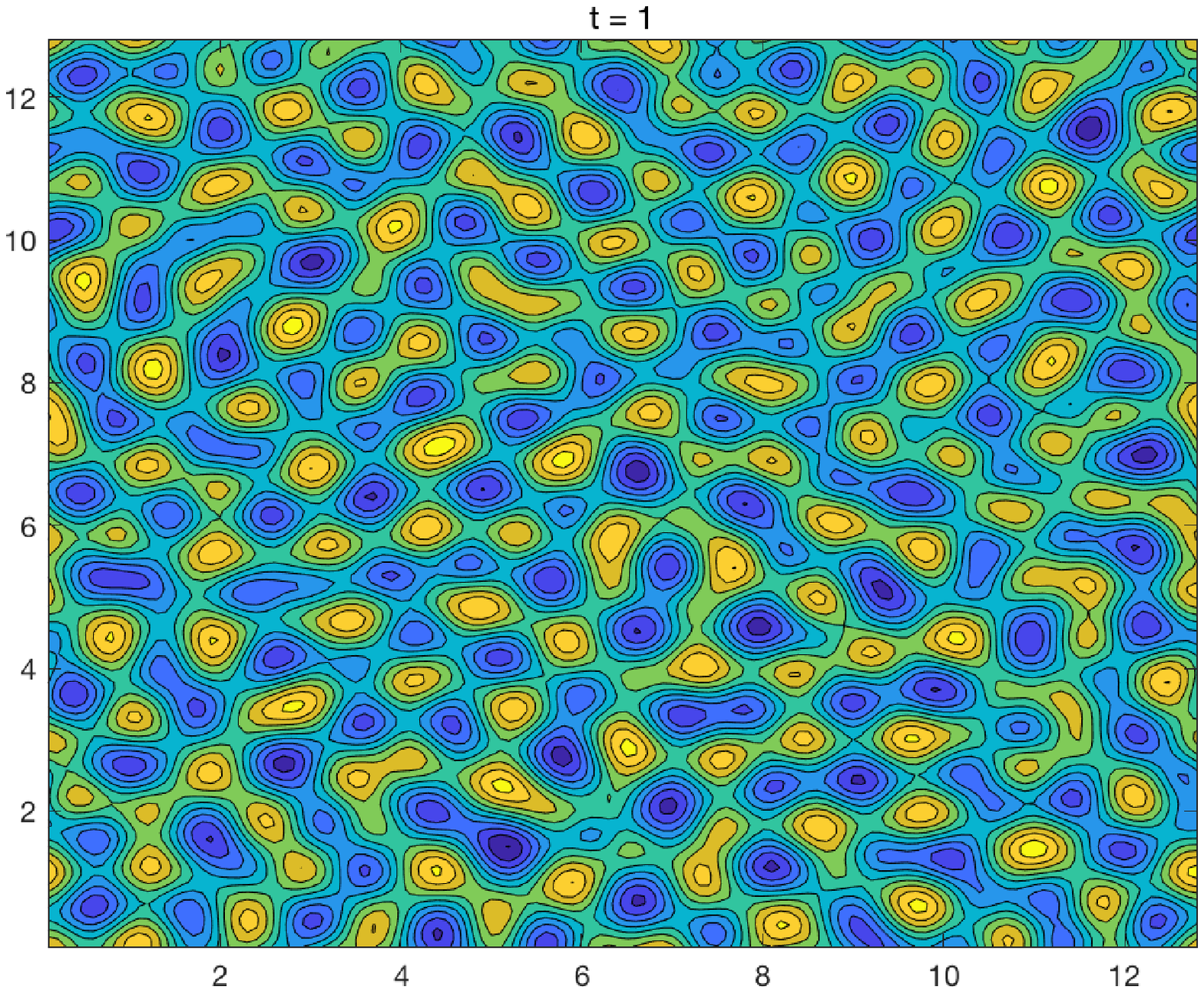}
 			\end{minipage}
 			\begin{minipage}{0.3\textwidth}
 				\includegraphics[width=\textwidth]{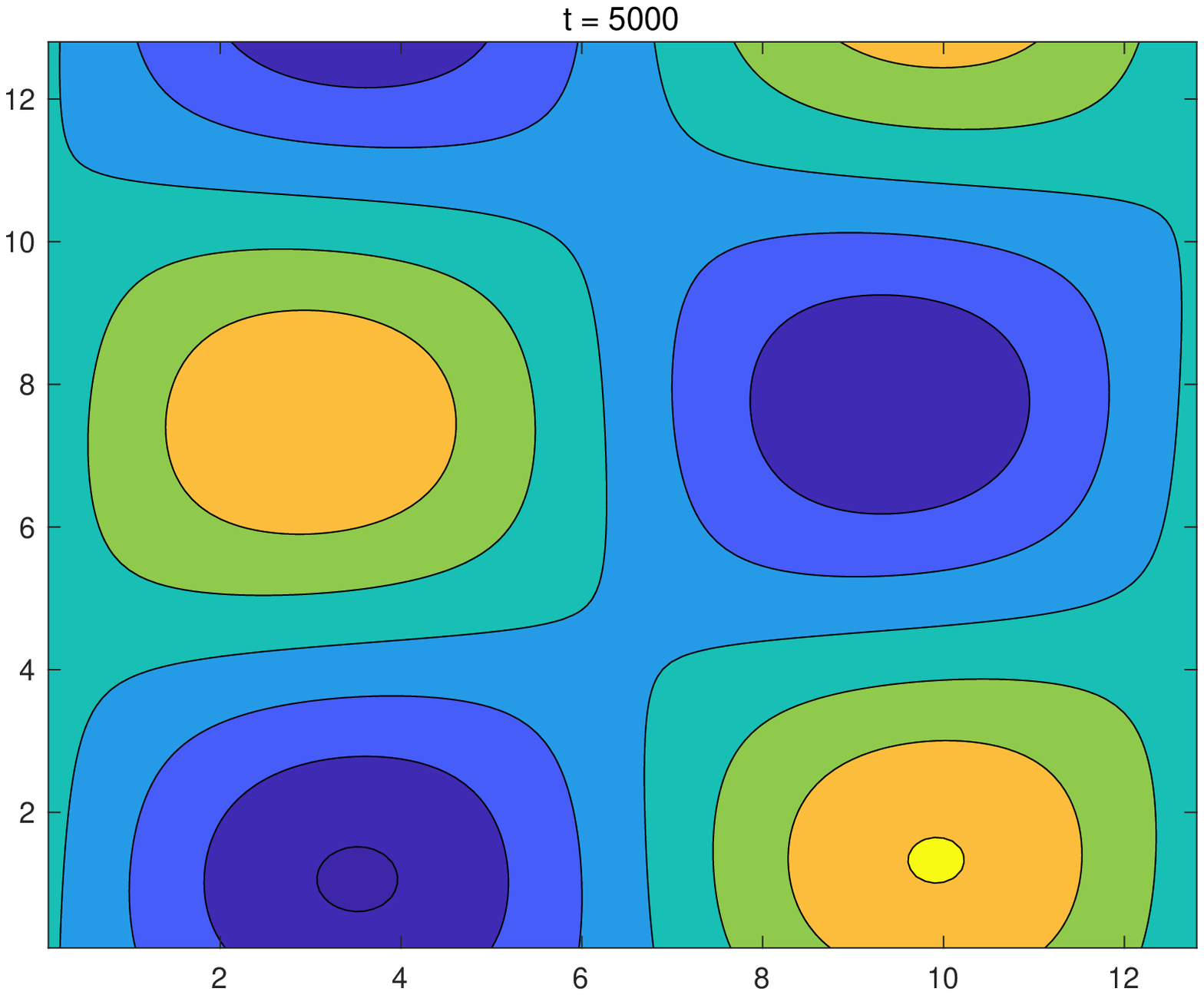}
 			\end{minipage}
 			\begin{minipage}{0.3\textwidth}
 				\includegraphics[width=\textwidth]{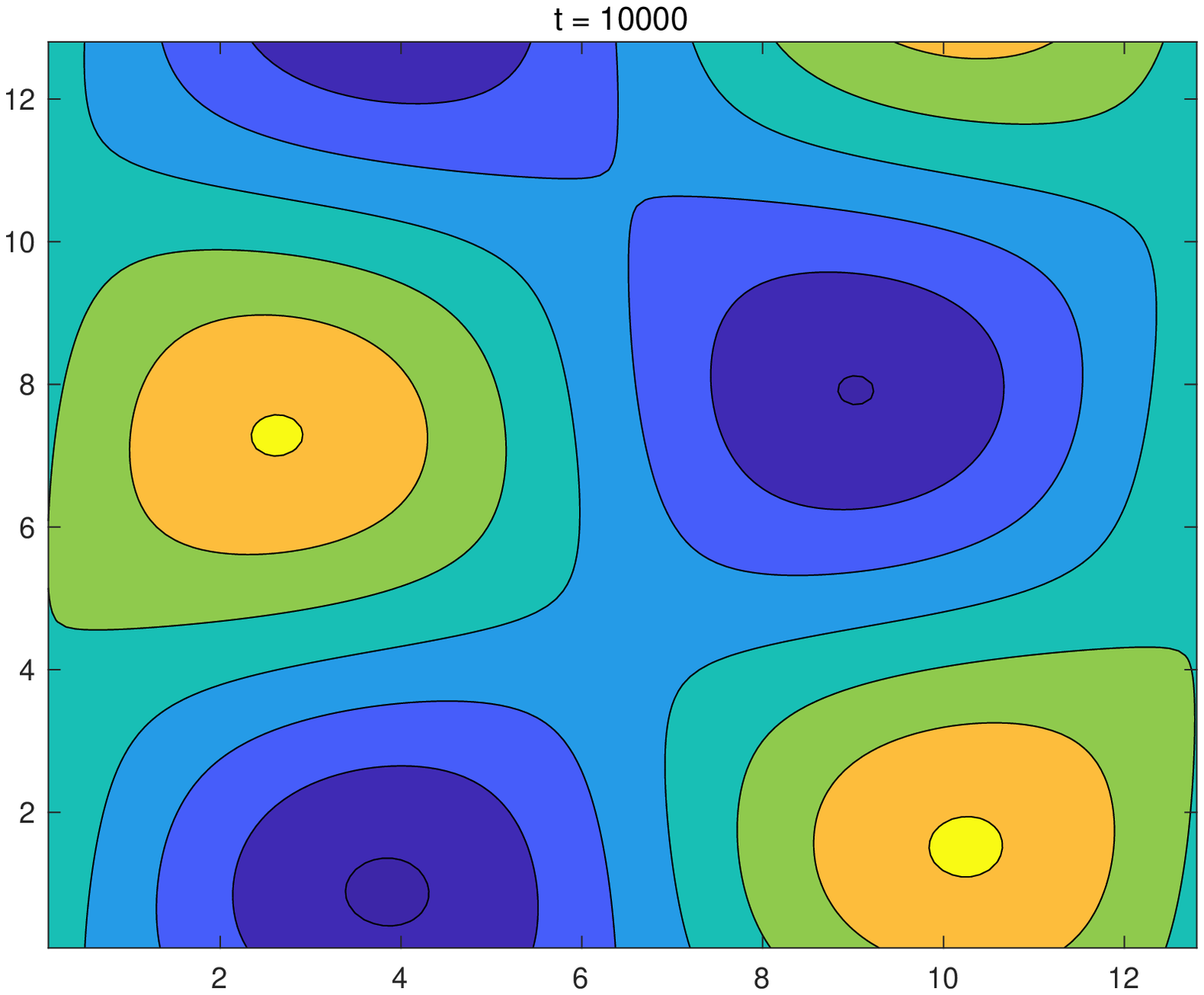}
 			\end{minipage}
 		}
 		\noindent\makebox[\textwidth][c] {
 			\begin{minipage}{0.3\textwidth}
 				\includegraphics[width=\textwidth]{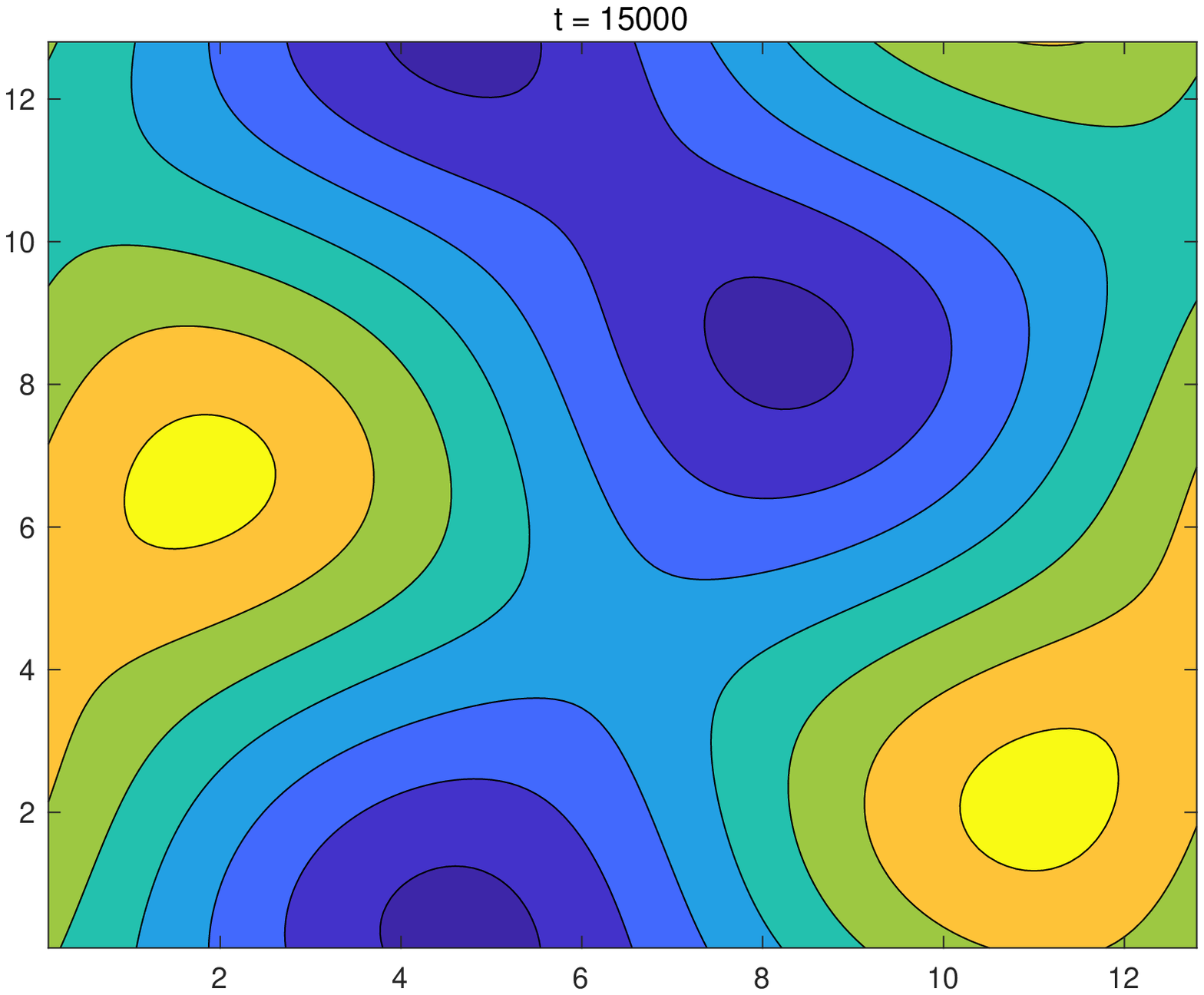}
 			\end{minipage}
 			\begin{minipage}{0.3\textwidth}
 				\includegraphics[width=\textwidth]{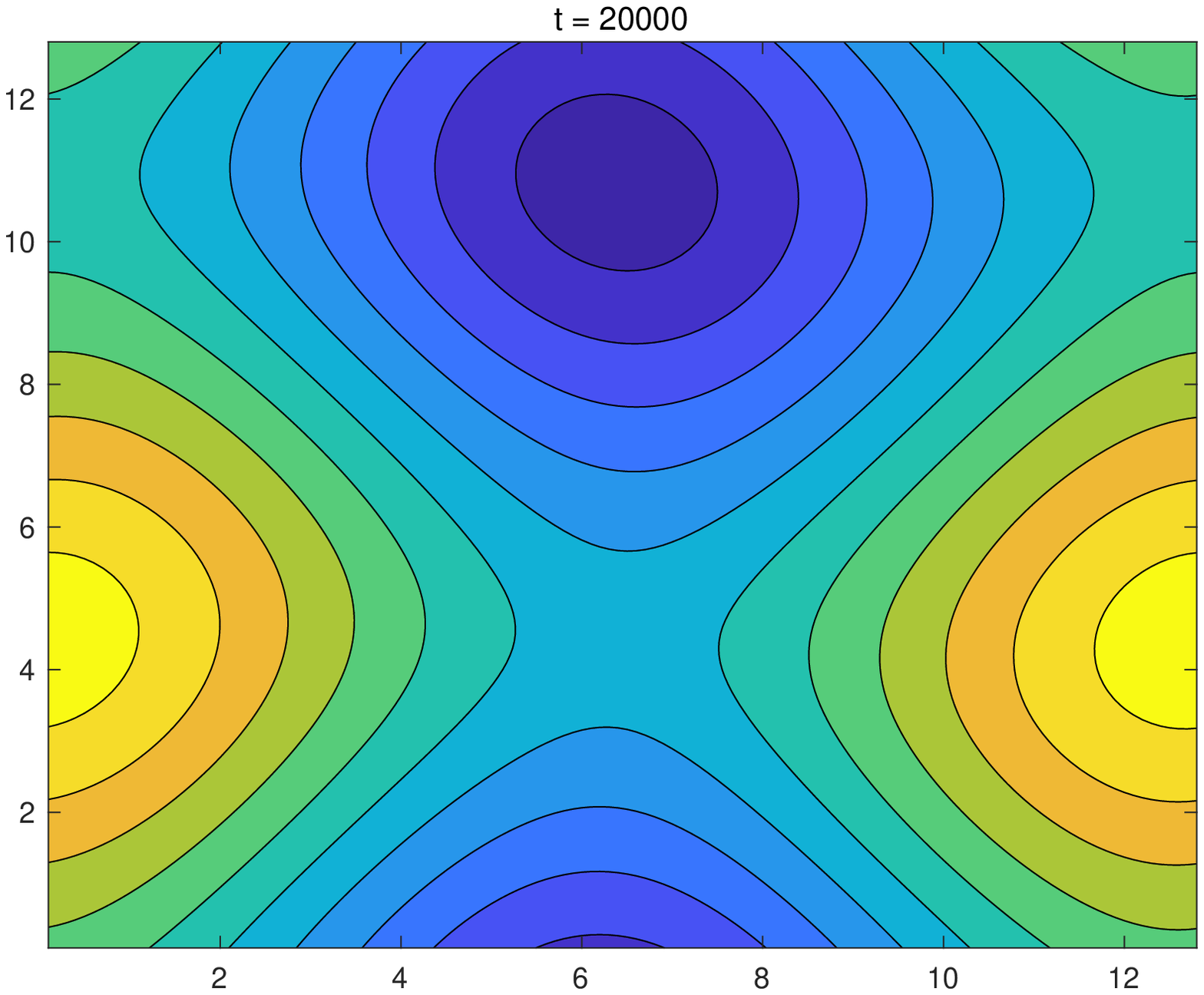}
 			\end{minipage}
 			\begin{minipage}{0.3\textwidth}
 				\includegraphics[width=\textwidth]{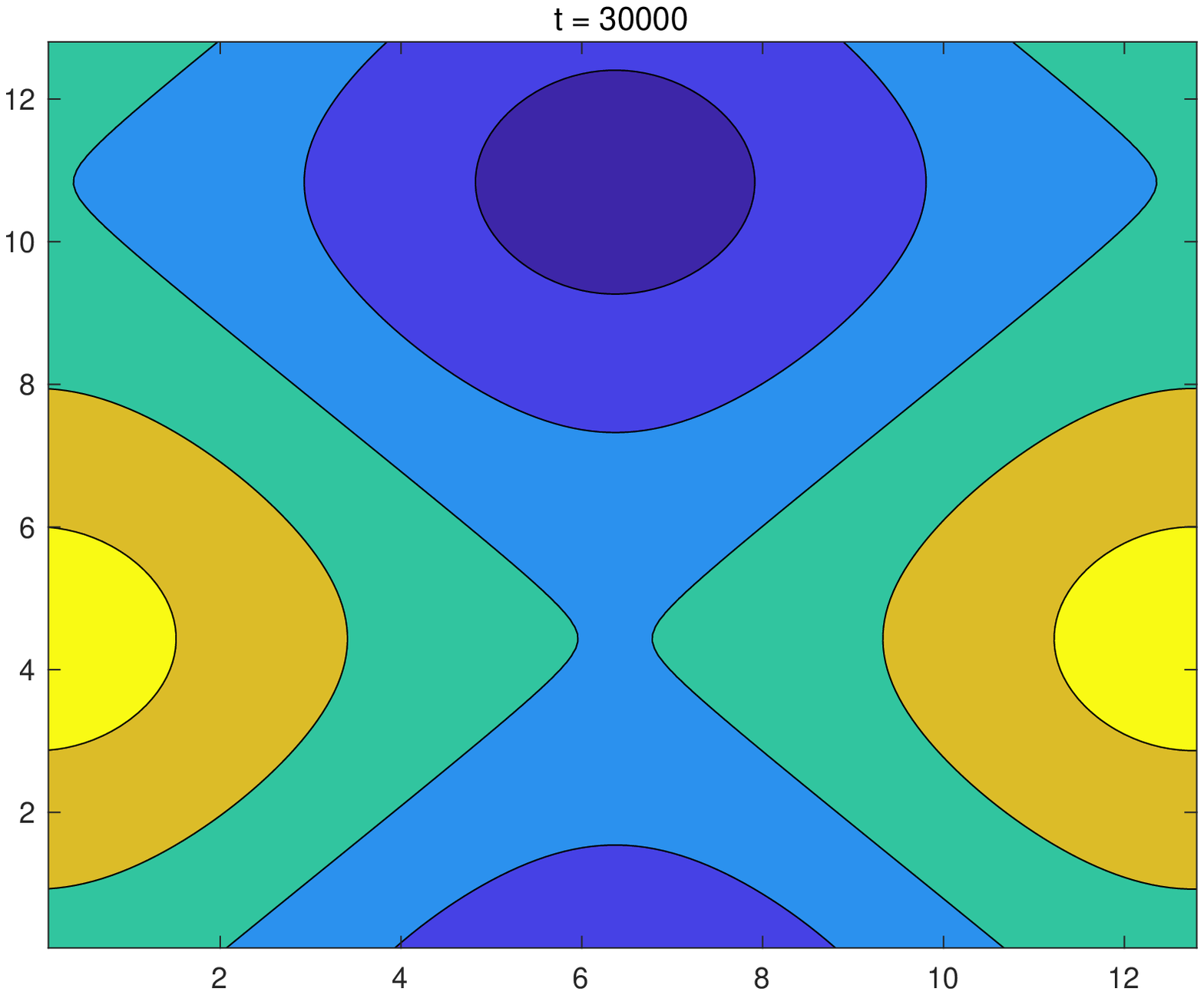}
 			\end{minipage}
 		}

 		\caption{Snapshots of the numerical solutions of scheme \eqref{eq:etdms nss} with $A=\frac{27\left(1+\bm{\overline{C}}_1\right)^4}{512}$ and variable time step-size: $\tau=10^{-6}$ for $t\le 1$, $\tau=10^{-5}$ for $1<t\le 10$, $\tau=10^{-4} $ for $10<t\le 100$ and $\tau=10^{-3}$ for $t>100$.}\label{fig: ss_1}
 	\end{figure}

  	\begin{figure}[ht]
 	\centering
 	\noindent\makebox[\textwidth][c] {
 		\begin{minipage}{0.3\textwidth}
 			\includegraphics[width=\textwidth]{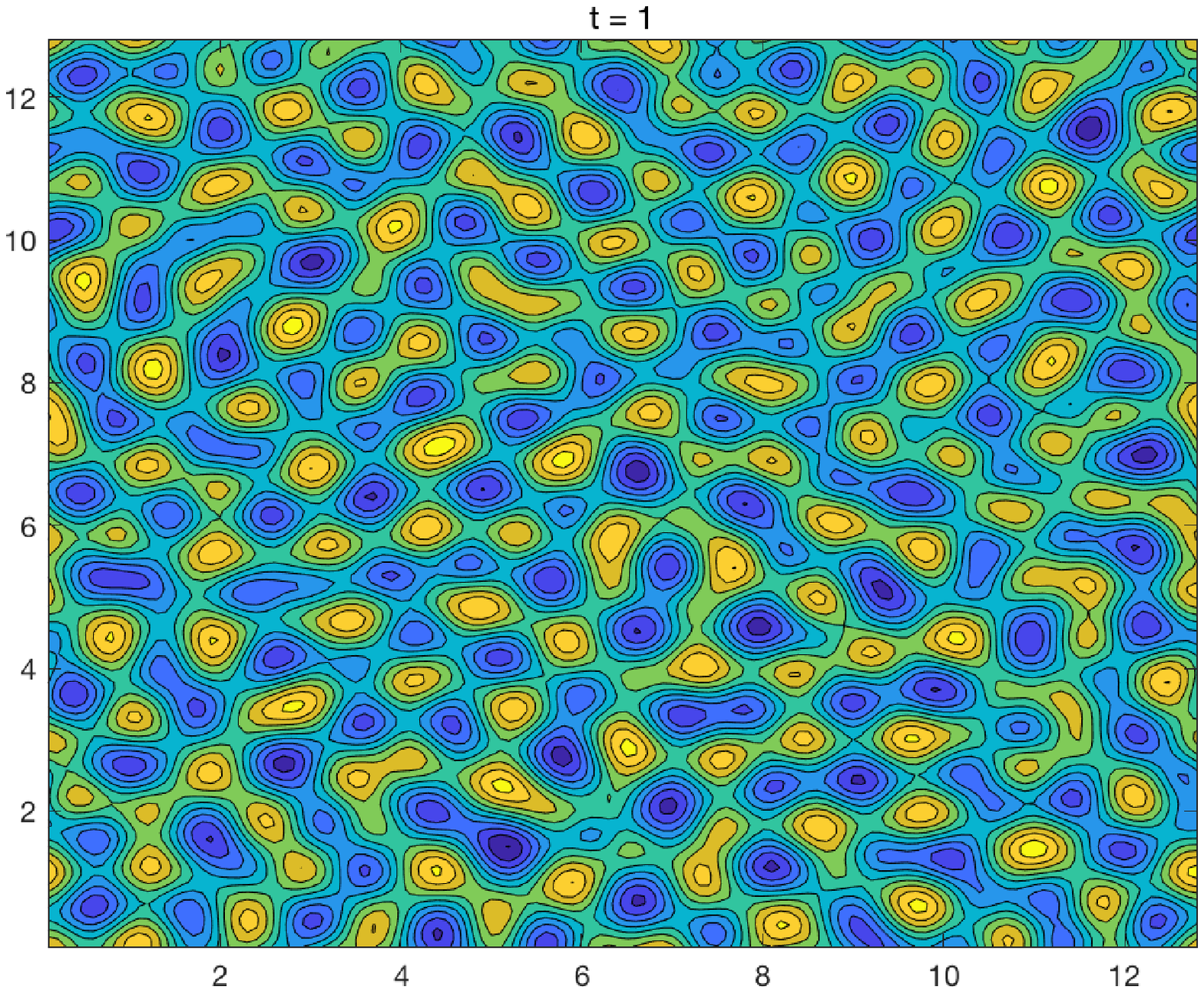}
 		\end{minipage}
 		\begin{minipage}{0.3\textwidth}
 			\includegraphics[width=\textwidth]{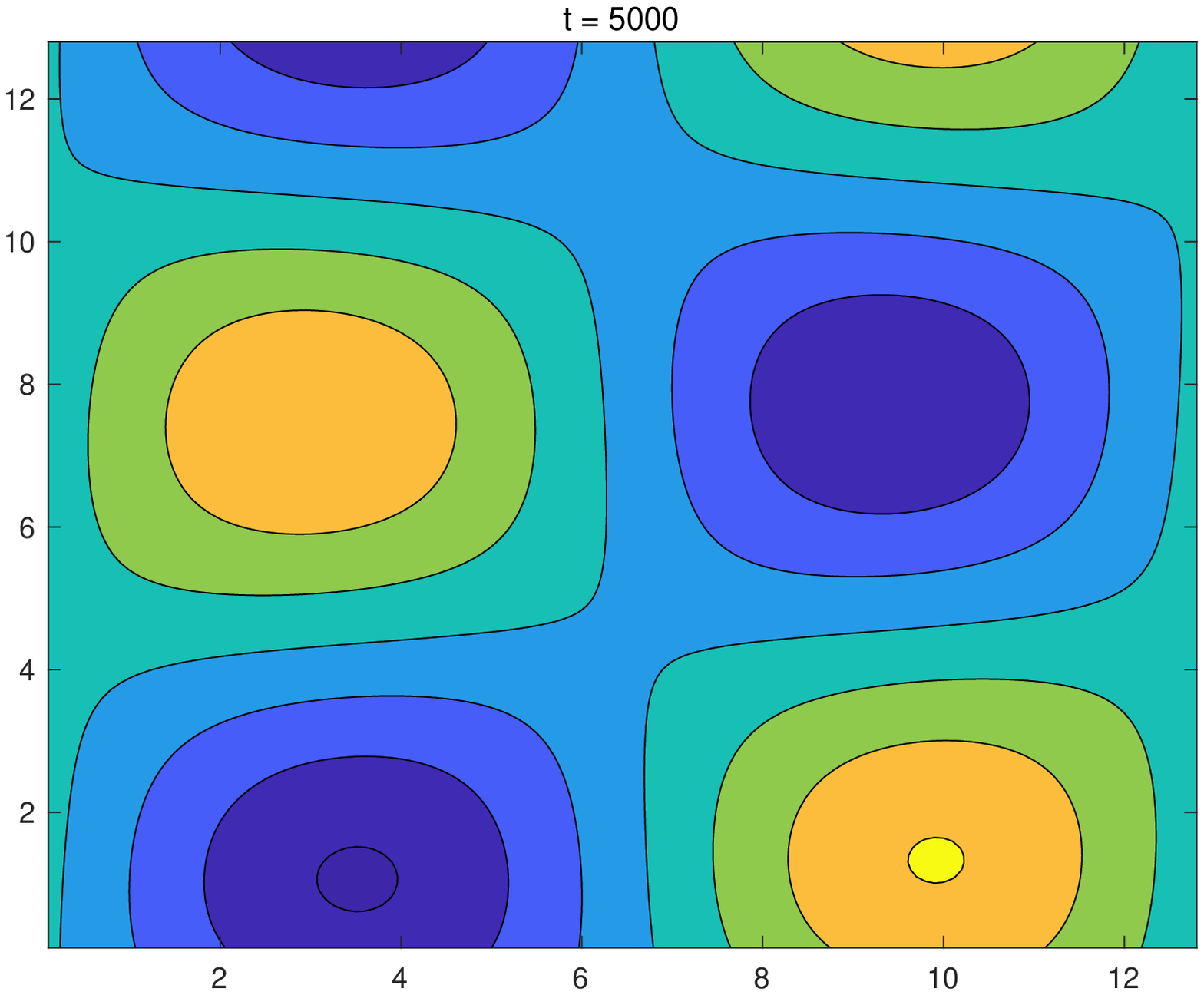}
 		\end{minipage}
 		\begin{minipage}{0.3\textwidth}
 			\includegraphics[width=\textwidth]{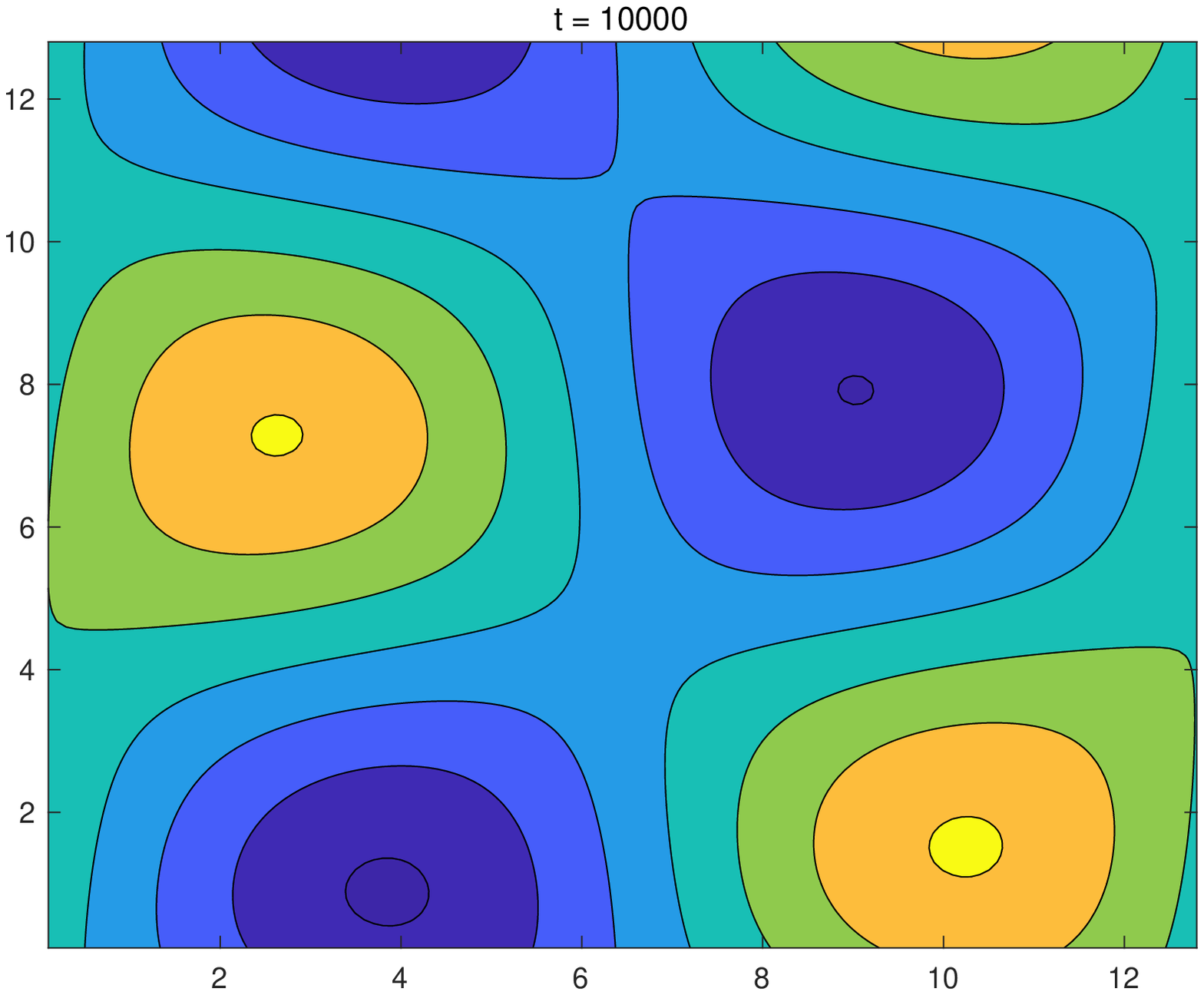}
 		\end{minipage}
 	}
 	\noindent\makebox[\textwidth][c] {
 		\begin{minipage}{0.3\textwidth}
 			\includegraphics[width=\textwidth]{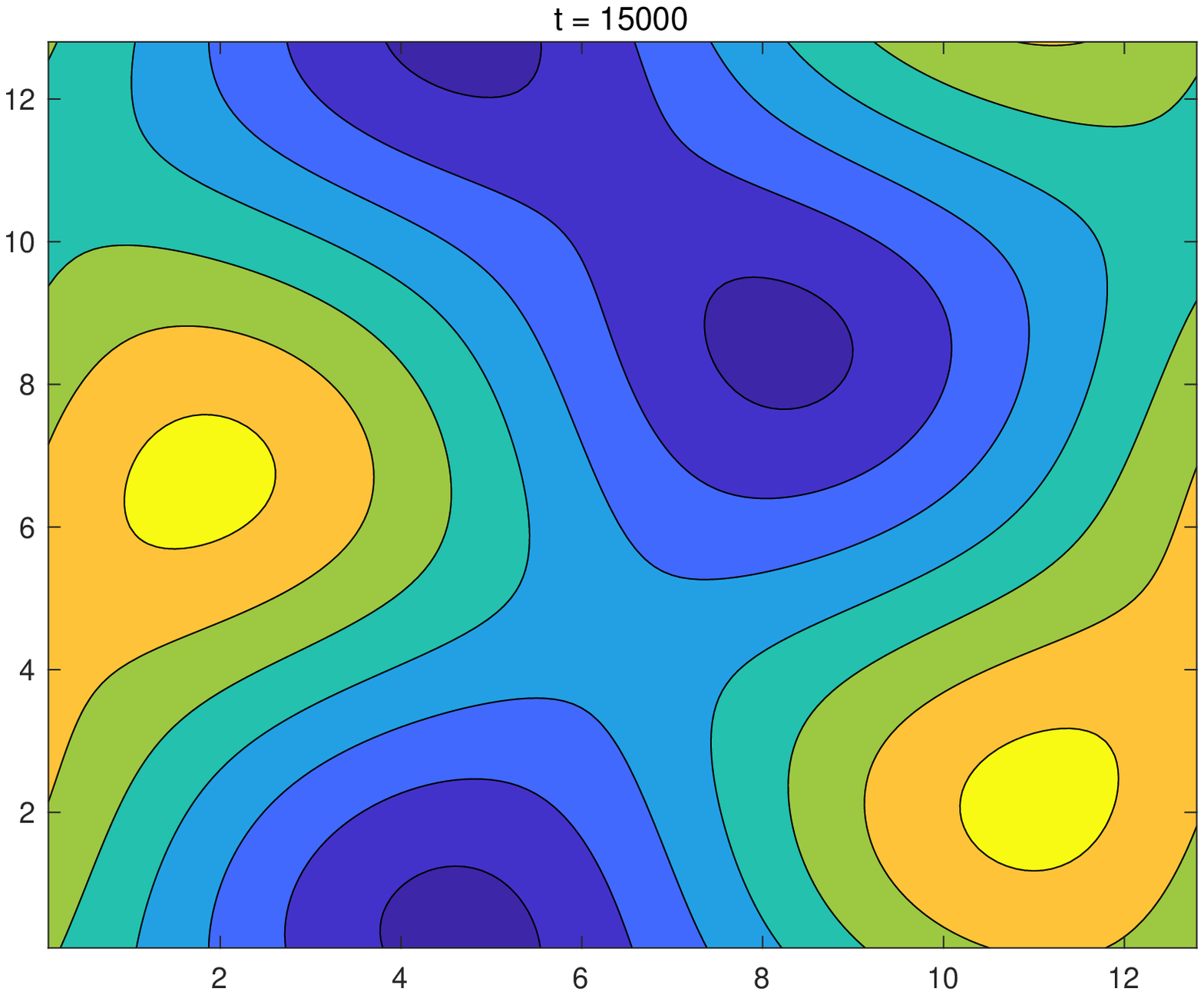}
 		\end{minipage}
 		\begin{minipage}{0.3\textwidth}
 			\includegraphics[width=\textwidth]{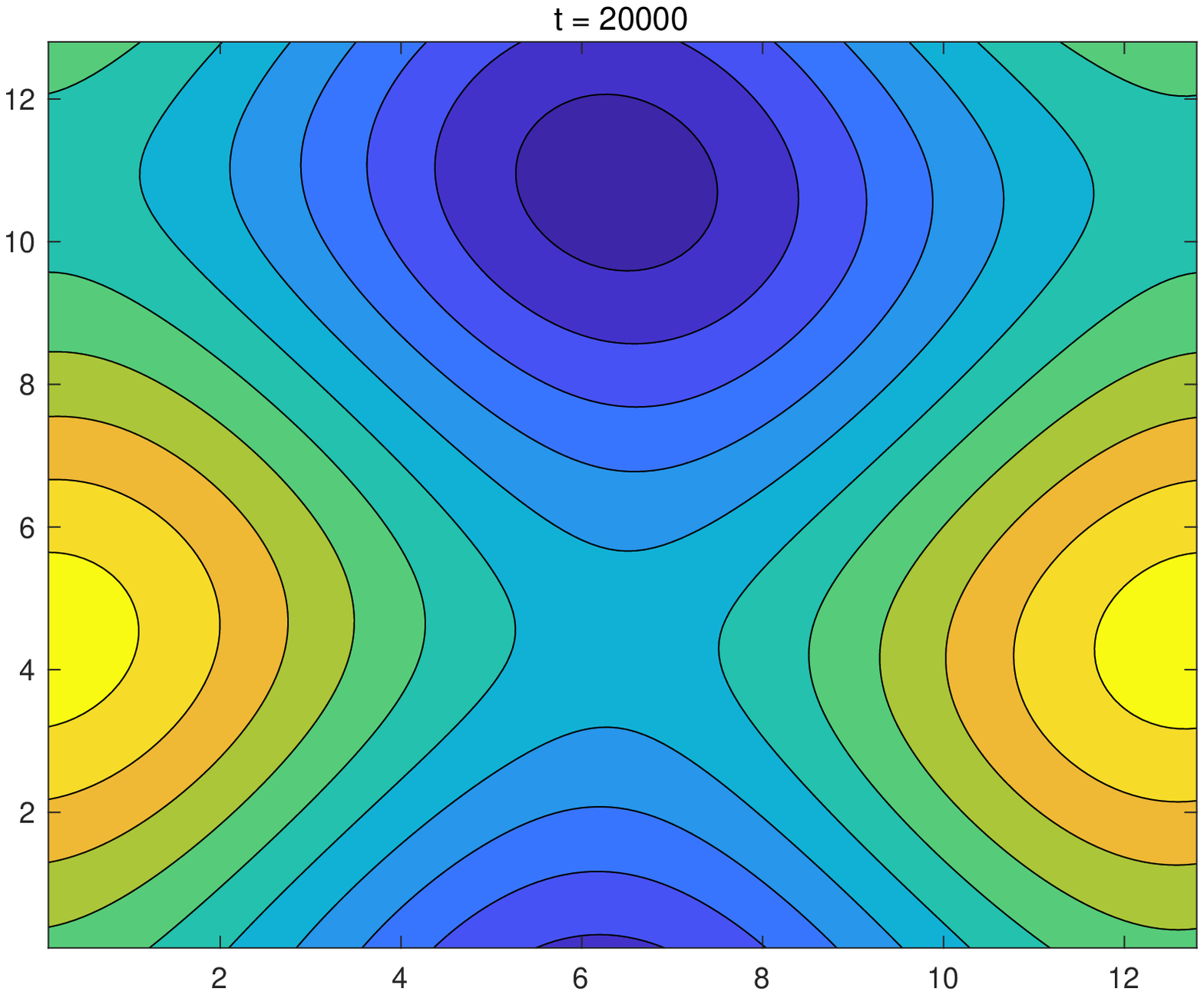}
 		\end{minipage}
 		\begin{minipage}{0.3\textwidth}
 			\includegraphics[width=\textwidth]{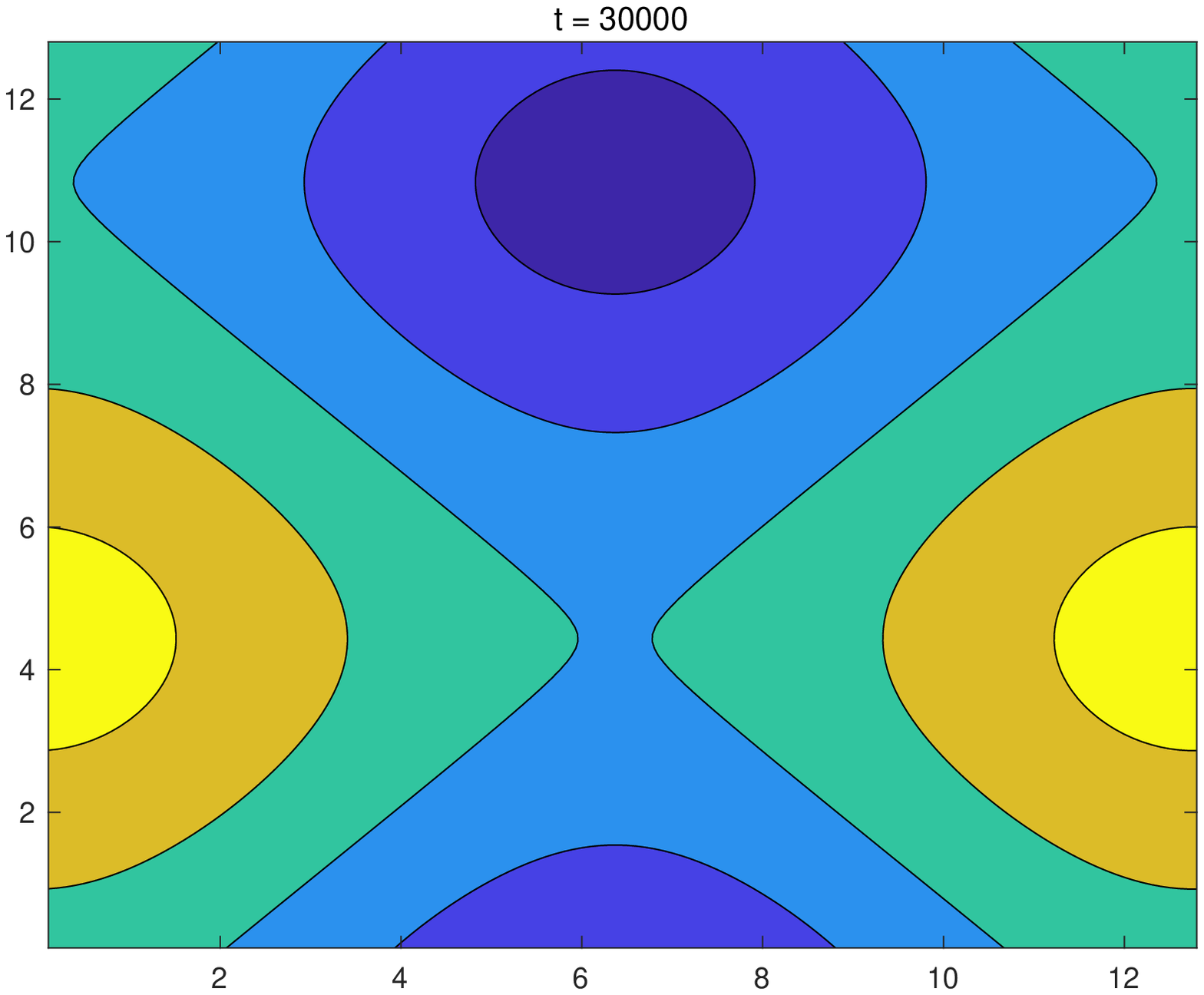}
 		\end{minipage}
 	}
 	\caption{Snapshots of the numerical solutions of scheme \eqref{eq:etdms nss} with $A=10$ and variable time step-size: $\tau=10^{-6}$ for $t\le 1$, $\tau=10^{-5}$ for $1<t\le 10$, $\tau=10^{-4} $ for $10<t\le 100$ and $\tau=10^{-3}$ for $t>100$.}\label{fig: ss_2}
 \end{figure}

  	\begin{figure}[ht]
	\centering
	\noindent\makebox[\textwidth][c] {
		\begin{minipage}{0.3\textwidth}
			\includegraphics[width=\textwidth]{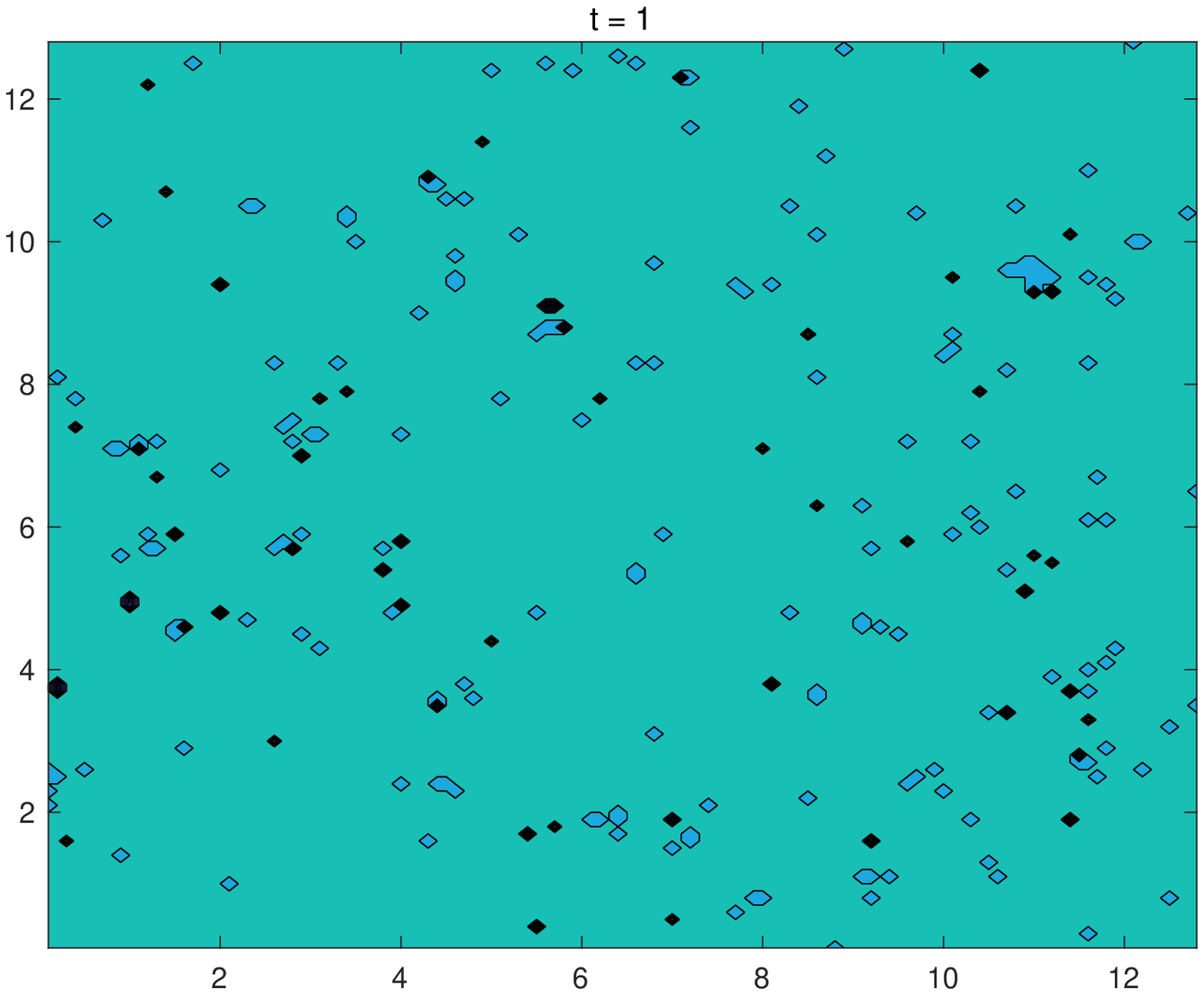}
		\end{minipage}
		\begin{minipage}{0.3\textwidth}
			\includegraphics[width=\textwidth]{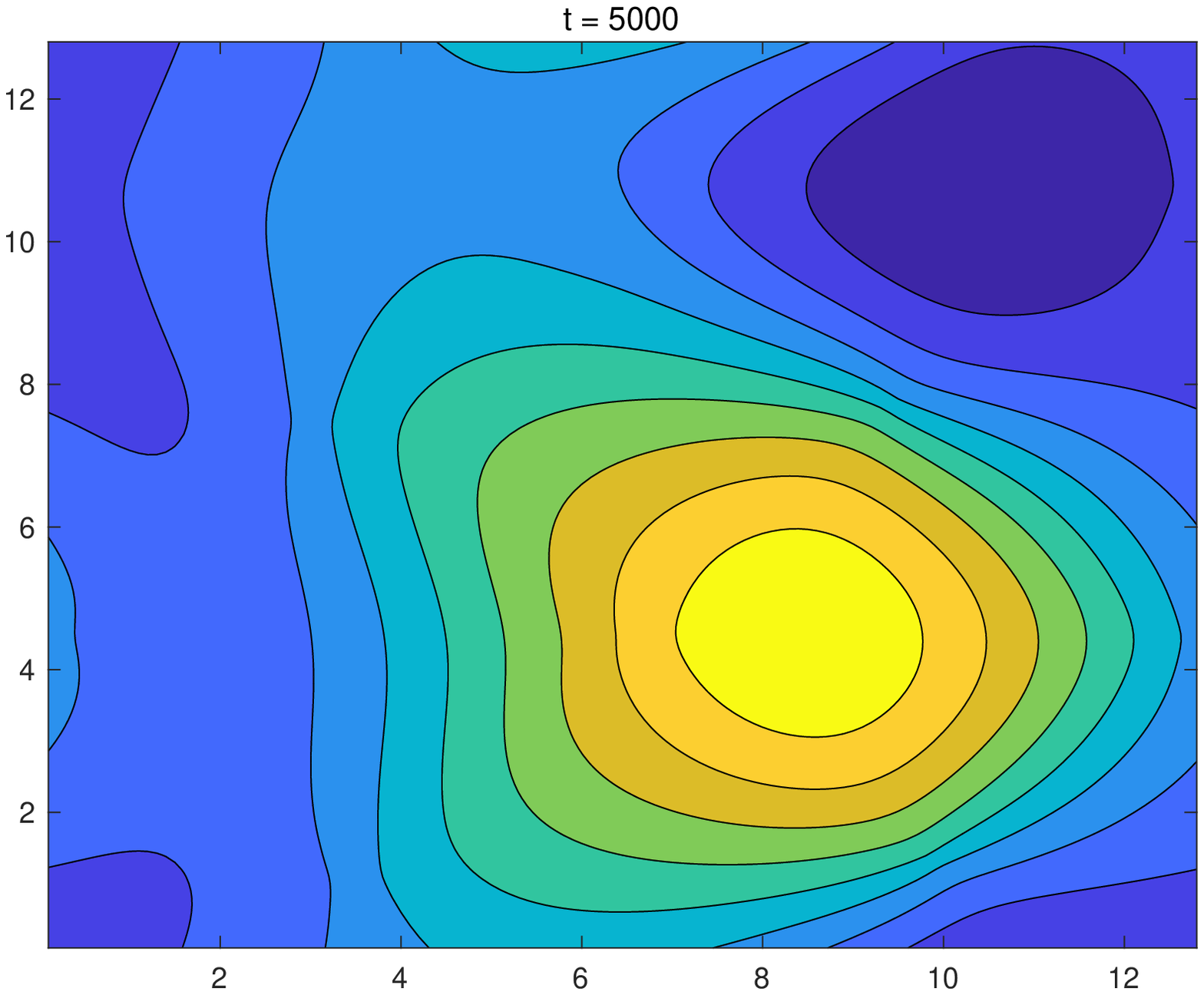}
		\end{minipage}
		\begin{minipage}{0.3\textwidth}
			\includegraphics[width=\textwidth]{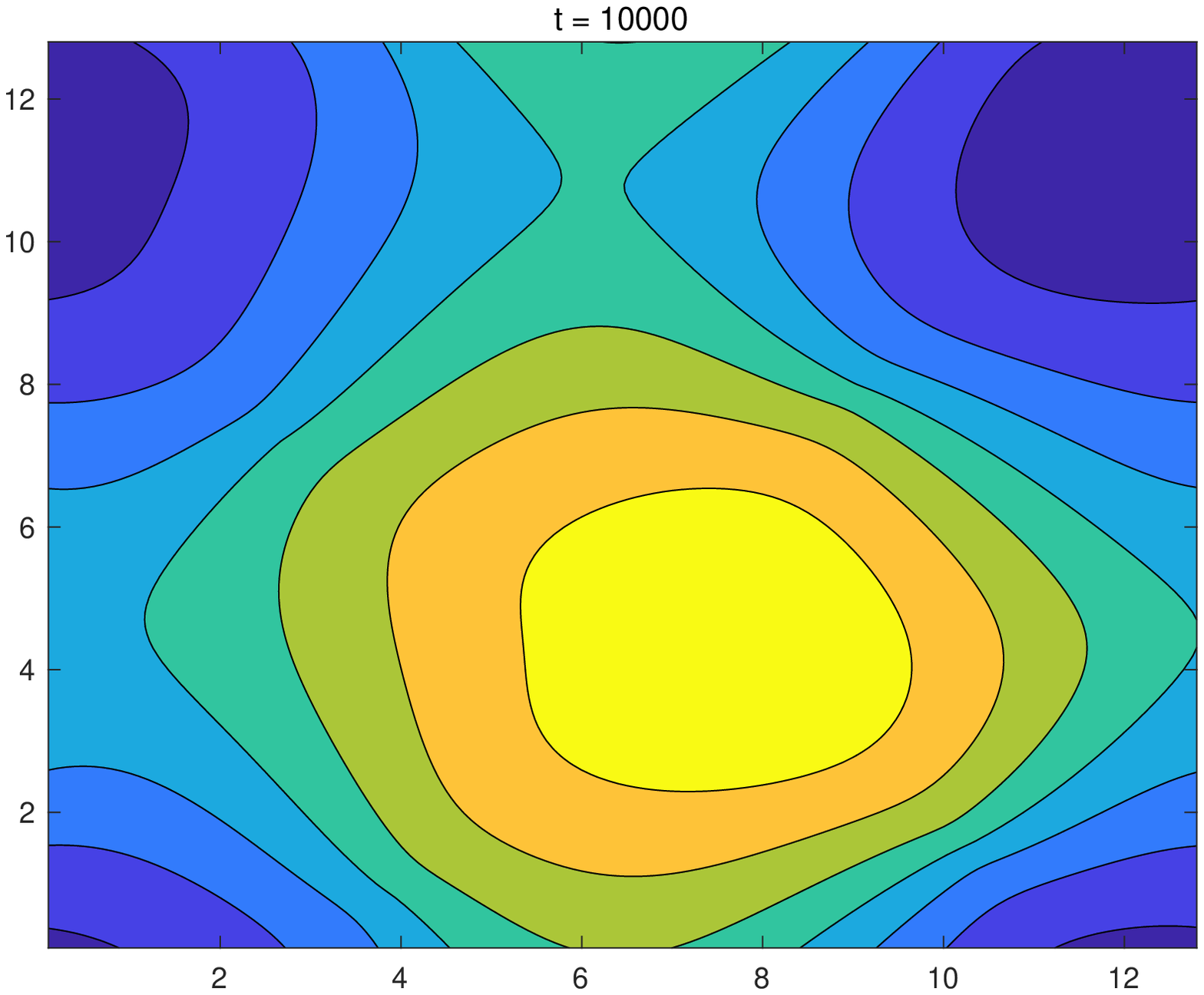}
		\end{minipage}
	}
	\noindent\makebox[\textwidth][c] {
		\begin{minipage}{0.3\textwidth}
			\includegraphics[width=\textwidth]{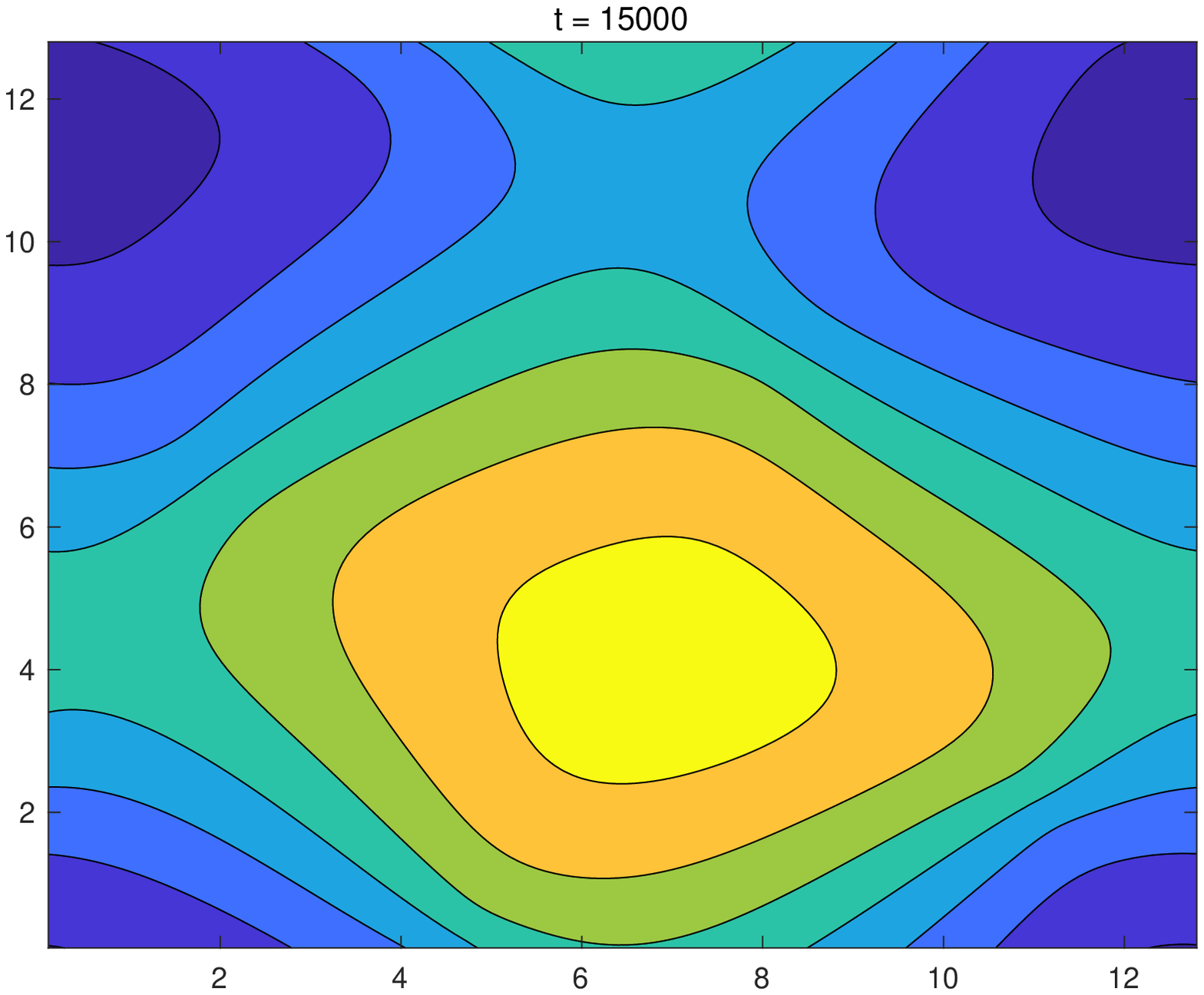}
		\end{minipage}
		\begin{minipage}{0.3\textwidth}
			\includegraphics[width=\textwidth]{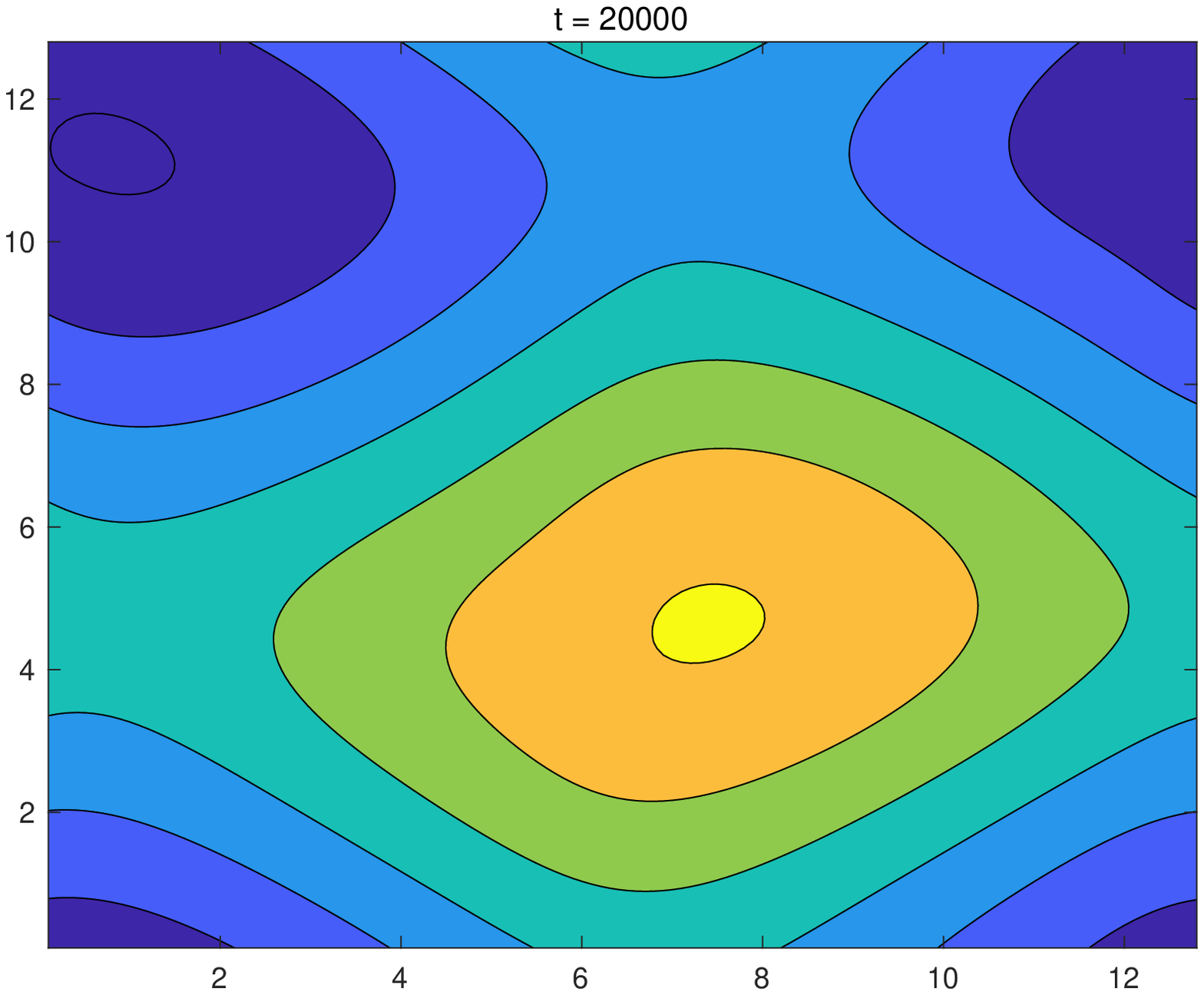}
		\end{minipage}
		\begin{minipage}{0.3\textwidth}
			\includegraphics[width=\textwidth]{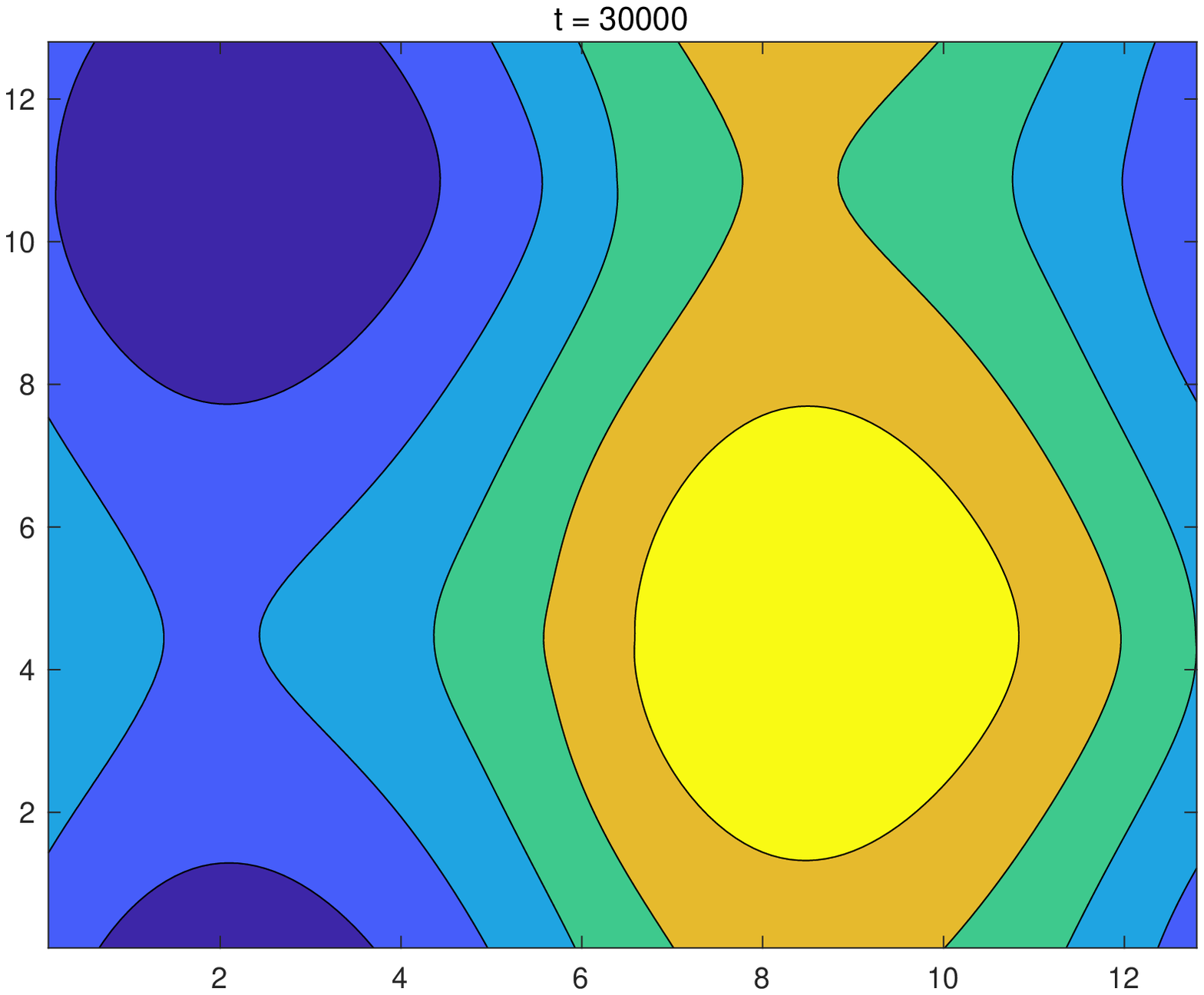}
		\end{minipage}
	}
	\caption{Differences of the numerical solutions of scheme \eqref{eq:etdms nss} between $A=10$ and  $A= \frac{27\left(1+\bm{\overline{C}}_1\right)^4}{512}$ at the same time level for variable time step-size: $\tau=10^{-6}$ for $t\le 1$, $\tau=10^{-5}$ for $1<t\le 10$, $\tau=10^{-4} $ for $10<t\le 100$ and $\tau=10^{-3}$ for $t>100$. The magnitudes of relative difference measured by discrete $L^2$-norm in these six snapshots are  $10^{-11}, ~10^{-5}, ~10^{-5}, ~10^{-4}, ~10^{-5},  ~10^{-5}$ respectively. The magnitude of numerical solution at time $t=1$ is $1$ and $100$ for the rest.} \label{fig: re}
\end{figure}

 \begin{figure}[ht]
 	\begin{center}
 		\includegraphics[width =0.8\textwidth]{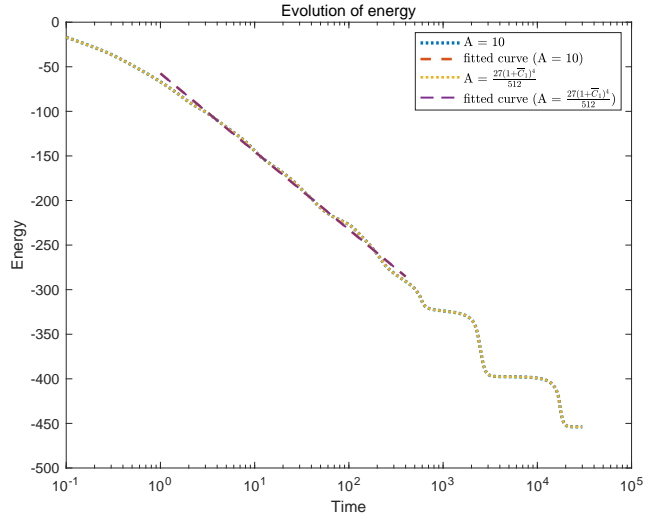}
 		\caption{Semi-log plot of the energy $E$ of scheme scheme \eqref{eq:etdms nss} with $\tau=10^{-6}$ for $t\le 1$, $\tau=10^{-5}$ for $1<t\le 10$, $\tau=10^{-4} $ for $10<t\le 100$ and $\tau=10^{-3}$ for $t>100$. Fitted line has the form $a \ln (t) + b$, with coefficients $a = -38.01$, $b = -57.48$ for both $A=10$ and $A= \frac{27\left(1+\bm{\overline{C}}_1\right)^4}{512}$.} \label{fig: energy1}
 	\end{center}
 \end{figure}

 \begin{figure}[ht]
 	\begin{center}
 		\includegraphics[width=0.8\textwidth]{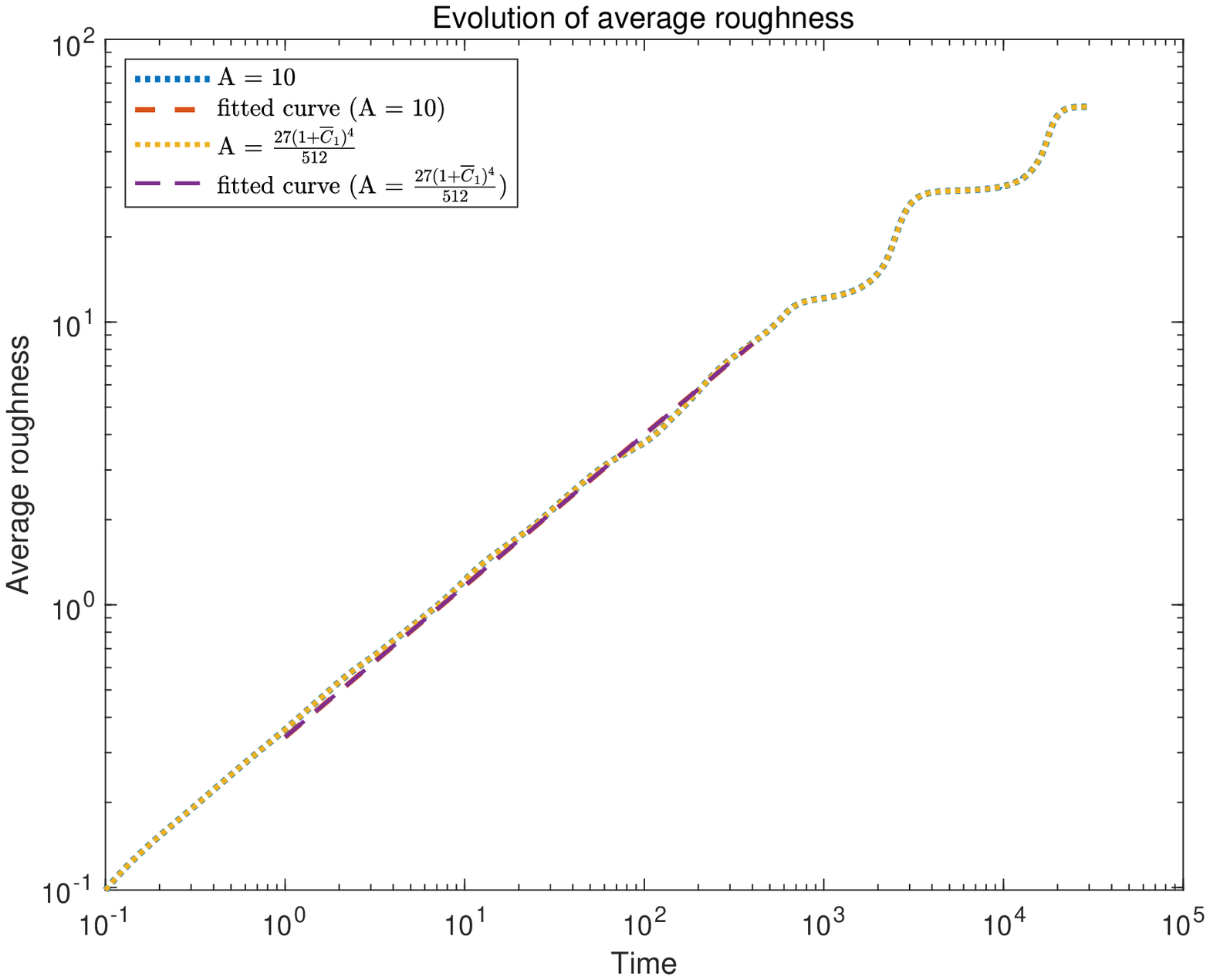}
 		\caption{The log-log plot of the average surface roughness $h$  of \eqref{eq:etdms nss} with $\tau=10^{-6}$ for $t\le 1$, $\tau=10^{-5}$ for $1<t\le 10$, $\tau=10^{-4} $ for $10<t\le 100$ and $\tau=10^{-3}$ for $t>100$. Fitted lines have the form $a t^b$, with coefficients $a = 0.3404$, $b = 0.5349$ for both $A=10$ and $A= \frac{27\left(1+\bm{\overline{C}}_1\right)^4}{512}$.} \label{fig: roughness1}
 	\end{center}
 \end{figure}
 
  \begin{figure}[ht]
 	\begin{center}
 		\includegraphics[width=0.8\textwidth]{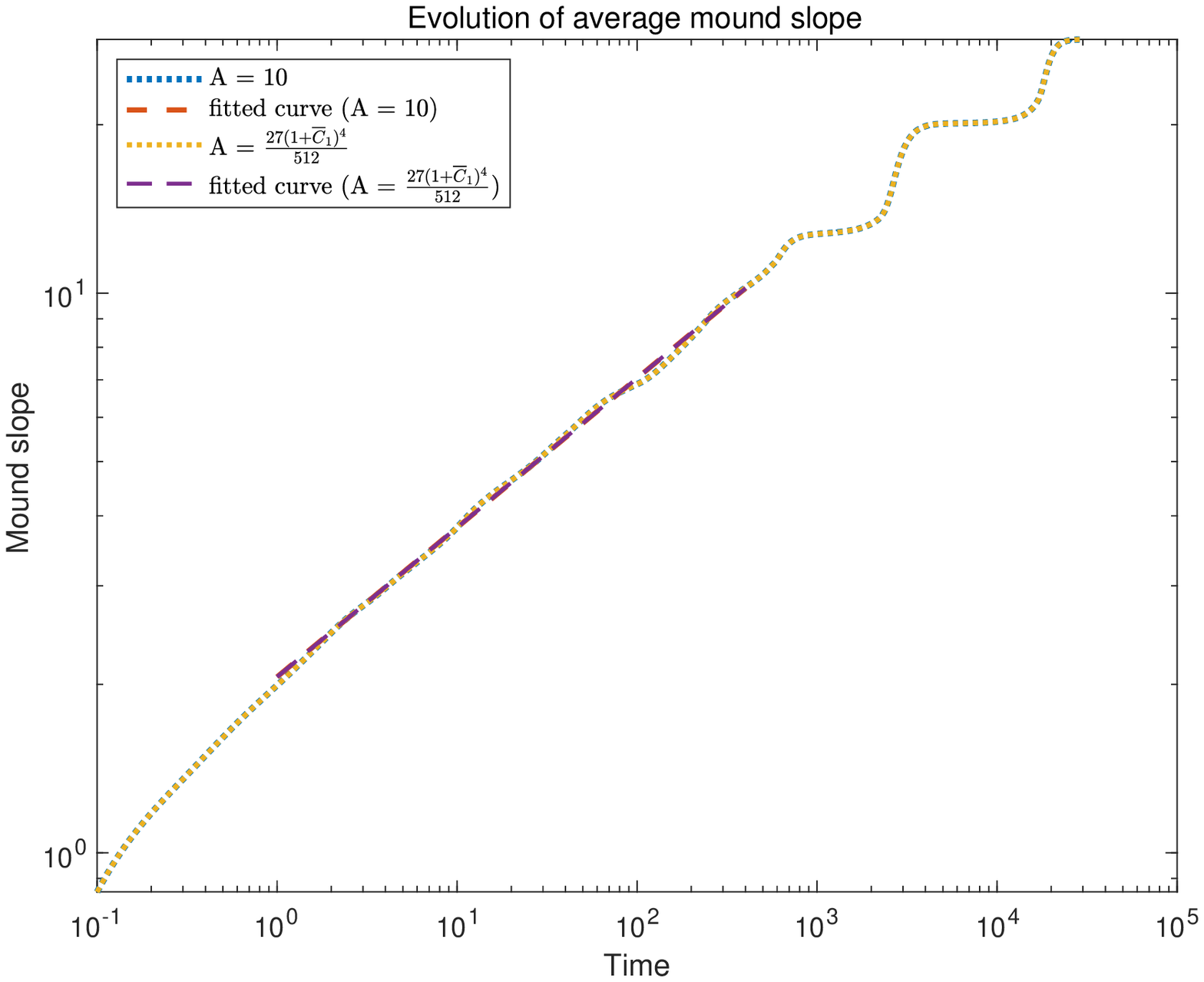}
 		\caption{The log-log plot of the average slope $m$ of \eqref{eq:etdms nss} with $\tau=10^{-6}$ for $t\le 1$, $\tau=10^{-5}$ for $1<t\le 10$, $\tau=10^{-4} $ for $10<t\le 100$ and $\tau=10^{-3}$ for $t>100$. Fitted lines have the form $a t^b$, with coefficients $a = 2.062$, $b = 0.2669$ for both  $A=10$ and $A= \frac{27\left(1+\bm{\overline{C}}_1\right)^4}{512}$.} \label{fig: slope1}
 	\end{center}
 \end{figure}

\section{Conclusions and remarks} \label{sec:conclusion}
A highly efficient $k^{th}$ order accurate and linear numerical scheme is proposed for gradient flows with mild nonlinearity by utilizing ETD and MS methods together with Lagrange interpolation and stabilization. The nonlinear term is treated explicitly and a $k^{th}$ order Dupont-Douglas type regularization $A\tau^k\frac{\partial }{\partial t}\cL^{p(k)} u$ is added for stability. The constant $A$ is independent of the small parameter $\epsilon$ or the discretization in space or in time. Lipschitz continuity for the nonlinear term is assumed. The instability caused by the explicit treatment can be totally overcome with the aid of regularization. The exponent $p(k)=\frac{(\beta+\gamma)k}{2}$ is determined explicitly by Lipschitz condition and the order of the scheme. As an example, a fourth order ETD-MS method is applied to a thin film epitaxy model without slope selection, some numerical experiments have been presented to validate the fourth order convergence in time and the unconditional long-time energy stability.
The method can be applied to a wide range of gradient flows after proper ``preparation'' of the original equation.

It is easy to see that our method can be generalized to the case of variable-step without any difficulty. All we need to do is to use a more general Lagrange interpolation polynomial. This opens a door to the development of higher order in time efficient time-adaptive strategies \cite{LTX16}. The Lagrange interpolation polynomials could be replaced by other appropriate interpolations so long as the near decomposition of the identity and the bounded properties are satisfied. The details together with the convergence analysis will be reported in a subsequent work.

	\section*{Acknowledgements}
	This work is supported in part by the following grants: NSFC12071090,  { Shanghai Science and Technology Research Program 19JC1420101} and a 111 project B08018 (W. Chen), NSFC11871159, Guangdong Provincial Key Laboratory for Computational Science and Material Design 2019B030301001 (X. Wang). The authors thank the anonymous reviewers for their comments and suggestions.

	\bibliographystyle{siam}
%	\bibliography{document.bib}

	\end{document}